\documentclass[12pt]{article}
\usepackage{amsthm}
\usepackage{amssymb,amsmath, txfonts}
\usepackage{graphicx}
\usepackage{float}
\usepackage{caption}
\usepackage{subcaption}
\usepackage{cite}
\usepackage{titlesec}
\usepackage{xcolor}
\begin{document}
	
	\def\abstractname{\bf Abstract}
	\def\dfrac{\displaystyle\frac}
	\let\oldsection\section
	\renewcommand\section{\setcounter{equation}{0}\oldsection}
	\renewcommand\thesection{\arabic{section}}
	\renewcommand\theequation{\thesection.\arabic{equation}}
	\newtheorem{theorem}{\indent Theorem}[section]
	\newtheorem{lemma}{\indent Lemma}[section]
	\newtheorem{proposition}[theorem]{Proposition}
	\newtheorem{definition}{\indent Definition}[section]
	\newtheorem{remark}{\indent Remark}[section]
	\newtheorem{corollary}{\indent Corollary}[section]
	\def\pd#1#2{\displaystyle\frac{\partial#1}{\partial#2}}
	\def\d#1{\displaystyle\frac{d#1}{dt}}
	
	\title{\LARGE\bf Self-similar Dynamics in the Critical $p$-Laplacian Patlak-Keller-Segel Model: Shrinking Blow-up and Expanding Propagation
		\author{Chunhua Jin, \quad Fengqing Zhang\thanks{
				Corresponding author. Email:  {\tt zhangfengqingmath@163.com}}
		\\
			\small (\it{School of Mathematical Sciences, South China
				Normal University, }
			\\
			\small \it{Guangzhou, 510631, China}\rm)
	}}
	
	\date{}
	
	\maketitle

\begin{abstract}
In this paper, we study the following Patlak-Keller-Segel model with $p$-Laplacian diffusion
\begin{align*}
			\left\{
			\begin{aligned}
				&\rho _t=\nabla \cdot \left( \left| \nabla \rho \right|^{p-2}\nabla \rho \right) -\chi \nabla \cdot \left( \rho \nabla c \right),
				\\
				&0=\varDelta c+\rho ^m,
			\end{aligned}\right.
		\end{align*}
and the exponent $m>0$ is chosen as
$$
m = \frac{(p-2)N + p}{N}.
$$
This relation ensures the scale invariance of the system and is conjectured to be the critical exponent that separates global boundedness from finite-time blow-up. We prove that, at the critical threshold $m=\frac{(p-2)N + p}{N}$, the system indeed admits finite-time blow-up solutions. More precisely, in the slow diffusion regime $p>2$, there exist backward self-similar blow-up solutions that are radially decreasing, compactly supported, and concentrate into a Dirac $\delta$-measure at the blow-up time $T$; and their supports shrink toward the origin at the rate $(T-t)^{\frac1{mN}}$. For the fast diffusion case $1<p\le 2$, we show that there are no backward self-similar blow-up solutions with finite-mass. Additionally, we also explore  forward self-similar solutions in both the slow diffusion and fast diffusion cases. These solutions also carry finite mass and exhibit a Dirac $\delta$-singularity at the initial moment. Specifically, in the slow diffusion case, the support expands at the rate $t^{\frac1{mN}}$, whereas in the fast diffusion case, the solution becomes strictly positive for all positive times.

Our work provides the first blow up analysis for the $p$-Laplacian Keller-Segel system when $p\ne 2$, and it confirms that the exponent $m$ given above is indeed the sharp threshold between global existence and finite time singularity formation.
\end{abstract}
	
	
\section{Introduction}

Since the 1970s, the Patlak-Keller-Segel system has served as a fundamental model for describing chemotaxis-driven aggregation in biological systems, such as cellular slime molds. In its classical form, this system is presented as
\begin{align}
		\label{1-1}\left\{
		\begin{aligned}
			&	\rho _t=\Delta \rho -\chi \nabla \cdot \left( \rho \nabla c \right) ,  &&
			\\
			&	\tau c_t=\Delta c -\mu c +\rho,
		\end{aligned}\right.
	\end{align}
where $\tau\in \{0, 1\}$, $\rho$ denotes cell density and $c$ denotes the chemoattractant concentration.
When $\tau = 0$, the system simplifies to a parabolic-elliptic structure, it corresponds to the assumption that the chemical substance diffuses much faster than the species itself \cite{13}, which can also be regarded as a model of self-attracting particles \cite{BH, Wo}.
Over the past several decades, extensive research has been conducted on the dynamic behaviors exhibited by the Keller-Segel system, including global existence, blow-up, and critical mass phenomena etc. \cite{4, 16, 18, 29, 30, 32, 34, 36, 42}. Notably, in two dimensions, a critical mass
$M^*$ exists such that solutions exist globally if the initial mass is below $M^*$, and may blow up in finite time if the mass exceeds $M^*$. While in higher dimensions $N \geq 3$, there is no critical mass phenomena, and for any small initial mass, there may exists solution blow up at a finite time.
On the other hand, the profile of blow-up solutions has also been extensively studied, with many works focusing on type I and type II blow-up solutions. Generally speaking, we say that the blow-up is of type I if
	$$\underset{t\rightarrow T}{\limsup}\,\,(T-t)\vert\vert\rho(\cdot,t)\vert\vert_{L^{\infty}(\Omega)}<\infty.$$
    Otherwise, the blow-up is of type II. It is remarkable that any blow-up solutions are
of type II for $N=2$, see \cite{21, 24, 25}, and both type I and type II blow-up solutions have been constructed in higher dimensions $N \geq 3$ \cite{6, 8, 12, 14, 15, 27, 28, 33, 35}.

In particular, cell migration more closely resembles motion in a porous medium, exhibiting nonlinear diffusion characteristics, that is, the migration rate depends on the cell density or its gradient. Therefore, chemotaxis models that incorporate nonlinear diffusion features, such as Newtonian flow ($\Delta\rho^m$) or non-Newtonian flow ($\nabla \cdot(| \nabla \rho|^{p-2}\nabla \rho)$), have become a new research focus. When the cell migration rate depends on density, the following porous medium-type diffusion model is often considered,
	\begin{align}
		\label{1-3}\left\{
		\begin{aligned}
			&		\rho _t=\Delta \rho ^m-\chi \nabla \cdot \left( \rho \nabla c \right),
			\\
			&	\tau c_t=\Delta c-\mu c+\rho .
		\end{aligned}\right.
	\end{align}
 It is well-established that a critical exponent $m_N =\frac{2(N-1)}{N}$ ($N\ge 3$) exists, which sharply distinguishes between global existence and finite-time blow-up of solutions. Specifically,

$\bullet$ When $m> m_N$, global solvability holds without any size restrictions on initial data for  $\tau \in \{0,1\}$ and  $\mu = 1$ \cite{40}.

$\bullet$ For  $m \leq m_N$, solutions may exhibit blow-up under certain conditions, e.g., finite-time blow-up occurs for  $1 < m < m_N$  with negative initial free energy \cite{1}, or for  $1 < m \leq m_N$  under certain initial data and specific parameter ranges \cite{7, 41}.

$\bullet$ At the critical threshold  $m = m_N$ , a critical mass phenomenon arises \cite{3}, and finite-mass self-similar blow-up solutions exist \cite{5} for  $N \geq 3$.

However, when migration is driven by density gradients,  the following $p$-Laplacian diffusion model emerges naturally,
\begin{align}
		\label{1-4}\left\{
		\begin{aligned}
			&\rho _t=\nabla \cdot \left( \left| \nabla \rho \right|^{p-2}\nabla \rho \right) -\chi \nabla \cdot \left( \rho \nabla c \right), &&
			\\
			&\tau c_t=\Delta c-\mu c+\rho . &&
		\end{aligned}\right.
	\end{align}	
Existing studies have revealed the significant role of $p$ in the behavior of solutions:

$\bullet$  For $\tau =\mu =0$ \cite{9}, or $\tau =\mu =1$ \cite{23}, globally bounded weak solutions exist for  $1<p\le \frac{3N}{N+1}$ with small initial data, and for  $p > \frac{3N}{N+1}$  with arbitrary initial data.

 Despite these advancements, comparing with the system \eqref{1-3}, a fundamental gap remains: {\bf no blow-up theory has been developed for the case  $p \neq 2$. While global existence results are available, the critical exponent separating blow-up from global existence has not been identified for general $p$. This lack of a complete picture motivates the present work}.

In this paper, we focus on the following  Patlak-Keller-Segel (PKS) system with $p$-Laplacian diffusion:
\begin{align}
		\label{1-5}\left\{
		\begin{aligned}
			&\rho _t=\nabla \cdot \left( \left| \nabla \rho \right|^{p-2}\nabla \rho \right) -\chi \nabla \cdot \left( \rho \nabla c \right), &&
			\\
			&c(x,t)=(K*\rho^m)(x,t),&&
		\end{aligned}\text{in}\ \mathbb R^N\times (0,+\infty),\right.
	\end{align}
with $\chi >0$, $p>1$,
$$
K(x)=\left\{\begin{aligned}
&-\frac12|x|, && N=1;
\\
&-\frac1{2\pi}\ln|x|, &&N=2;
\\
&\frac{1}{(N-2)\omega_N|x|^{N-2}}, &&N\ge 3,
\end{aligned}\right.
$$
where $\omega_N$ represents the surface area of the unit sphere.
So that,
$$
\nabla c={\bar K}*\rho^m, \bar K=\nabla K=-\frac{x}{\omega_N|x|^{N}},
$$
with $\omega_1=2$, $\omega_2=2\pi$.
Here $c$ solves the  Poisson equation
$$-\Delta c=\rho^m, \ \text{ in $\mathbb R^N$}.
$$
We specifically choose
$$m=\frac{(p-2)N+p}{N}>0,\quad \left(\Leftrightarrow\quad p=\frac{(m+2)N}{N+1}\right).
$$
The selection of this particular exponent $m$ is fundamental. It is hypothesized to be the critical exponent for this system, meaning it represents the threshold between global boundedness and finite-time blow-up. This conjecture is supported based on the following evidence:

$\bullet$ In the classical case $p=2$, the critical exponent is known to be $m=\frac{2}{N}$. In fact, Liu, Winkler et al. \cite{22, 43} showed that for a related model, the solutions blow up in finite time if $m>\frac{2}{N}$ and remain globally bounded if $m<\frac{2}{N}$.

$\bullet$ The works of Cong et al.\cite{9, 23}  suggest that global solutions exist when $m<\frac{(p-2)N + p}{N}$ (e.g., with $ m = 1 $), indirectly supporting $m = \frac{(p-2)N + p}{N}$ as the critical threshold.

$\bullet$ Mathematically, the system exhibits scale invariance if and only if $m$ takes this specific value. Scale invariance is a common feature in problems exhibiting critical phenomena, often associated with the existence of a sharp threshold for blow-up.

From the above results, we observe that the previous studies on Keller-Segel system with $p$-Laplacian diffusion have primarily focused on global existence, while the blow-up theory remains open.
The purpose of this paper is to prove that in the critical case   $m=\frac{(p-2)N + p}{N}$, the system \eqref{1-5} can actually generate solutions that blow up in finite time. Specifically, in the slow-diffusion case $p>2$, we establish the existence of backward self-similar blow-up solutions that are radially monotone and compactly supported. Such solutions concentrate into a Dirac $\delta$-singularity at the blow-up time, and their supports shrink toward the origin at the rate $(T-t)^{\frac1{mN}}$. While for the fast-diffusion case with $1<p\le 2$, there is not self-similar blow-up solutions with finite mass.  This offers the first blow-up analysis for the case $p\not=2$, and thereby confirms the conjectured role of $m$ as the critical exponent sharply separating global existence from finite-time blow-up.

Additional,  in the fourth section of the paper, we also investigate forward self-similar solutions. In both the fast  and slow diffusion regimes, we establish the existence of forward self-similar solutions with finite mass. All of these solutions exhibit a Dirac $\delta$-singularity at the initial time. In particular, for the slow diffusion case, the initial data concentrate at the origin, forming a Dirac $\delta$ measure,
and the support of the solution expands at the rate $t^{\frac{1}{mN}}$ as time increases. While, for the fast diffusion case, the solution becomes positive everywhere for any positive time.

At last, we also present a brief discussion on the sign structure of the chemical concentration in the parabolic-elliptic Keller-Segel model, examining its dimensional dependence and profound influence on aggregation dynamics, along with interpretations of negative concentrations in physical contexts.

\section{Preliminaries and main results: Backward Self-Similar Blow-up Solutions}

Noticing that the system exhibits scale invariance, we now proceed to construct backward self-similar blow-up solutions to the system \eqref{1-5}. To this end, we let
	$$\rho(x,t) =( T-t) ^{-\alpha}\phi ( r) , \quad c( x,t ) =( T-t ) ^{\gamma}\psi (r),$$
	where $r=( T-t ) ^{-\beta} |x|.$ By requiring the solution to be invariant under the similarity transformation, we obtain the scaling exponents
	\begin{align}\label{2-1}
		&\alpha =\frac{1}{m},\quad \beta=\frac{1}{mN},\quad \gamma = \frac{2-mN}{mN},
	\end{align}
with
	\begin{align}\label{2-2}
	m=\frac{( p-2 ) N+p}{N}>0\quad \left(\text{i.e.,} \ p=\frac{(m+2)N}{N+1}\right).
	\end{align}
A direct calculation shows that the mass of $\rho$ defined in this way is conserved and independent of time $t$, that is
	$$
	\int_{R^N}{\rho(x,t)}dx=\int_{R^N}{( T-t ) ^{-\alpha}\phi ( ( T-t ) ^{-\beta}   |x|   )}dx=\omega_N\int_0^\infty r^{N-1}\phi(r)dr \equiv M,
	$$
where $\omega_N$ is the surface area of the unit sphere.
Therefore, the system \eqref{1-5} is transformed into
		\begin{align}
		\label{2-4}\left\{
		\begin{aligned}
			&\frac{1}{m}\phi +\frac{r}{mN}\phi ^{\prime}=\left( \left| \phi ^{\prime} \right|^{p-2}\phi ^{\prime} \right) ^{\prime}+\frac{N-1}{r}\left| \phi ^{\prime} \right|^{p-2}\phi ^{\prime}-\chi \left( \phi \psi ^{\prime} \right) ^{\prime}- \frac{\chi(N-1)}{r}\left( \phi \psi ^{\prime} \right)
			,
			\\
			&\psi ^{''}+\frac{N-1}{r}\psi ^{\prime}+\phi ^m=0.
		\end{aligned}\right.
	\end{align}
	From the first equation of \eqref{2-4}, we derive that
	 \begin{align}
	 	\label{2-5}
	 	\begin{aligned}
	 		& \frac{r}{mN}=\frac{\left| \phi ^{\prime} \right|^{p-2}\phi ^{\prime}}{\phi }-\chi \psi ^{\prime}.
	 	\end{aligned}
	 \end{align}
Substituting \eqref{2-5} into the second equation of \eqref{2-4} yields
	 \begin{align*}
	 	\begin{aligned}
	 		& \left( \frac{\left| \phi ^{\prime}\left( r \right) \right|^{p-2}\phi ^{\prime}\left( r \right)}{\phi \left( r \right)} \right) ^{\prime}+\frac{N-1}{r}\frac{\left| \phi ^{\prime}\left( r \right) \right|^{p-2}\phi ^{\prime}\left( r \right)}{\phi \left( r \right)}+\chi \phi ^m\left( r \right) -\frac{1}{m}
	 		=0.
	 	\end{aligned}
	 \end{align*}
Therefore, finding the radial solutions to system \eqref{2-4} reduces to solving the following equation
\begin{align}
\label{2-7}\left\{
\begin{aligned}
&\left( \frac{\left| \phi ^{\prime}\left( r \right) \right|^{p-2}\phi ^{\prime}\left( r \right)}{\phi \left( r \right)} \right) ^{\prime}+\frac{N-1}{r}\frac{\left| \phi ^{\prime}\left( r \right) \right|^{p-2}\phi ^{\prime}\left( r \right)}{\phi \left( r \right)}+\chi \phi ^m\left( r \right) -\frac{1}{m}=0,
\\
&\phi\left( 0 \right) =A, \phi^{\prime}\left( 0 \right) =0.
\end{aligned}\right.
\end{align}
Here, the initial value condition $\phi^{\prime}\left( 0 \right)=0$ arise from the radial symmetry of the solution, and  $A$  is a positive constant to be determined.
Our primary objective is to determine whether a global nonnegative solution $\phi(r)$ exists, satisfying one of the following conditions:

     (i) $\displaystyle\lim_{r\to\infty}\phi(r) = 0$, following the definition of \cite{39}, such a solution is referred to as a {\bf ground state solution}.

     (ii)A compactly supported radial solution fulfilling a homogeneous Dirichlet-Neumann free boundary condition, i.e.,
\begin{align}\label{eq2-6}
\left\{\begin{aligned}
    			&\nabla\cdot\left(\frac{|\nabla\phi|^{p-2}\nabla\phi}{\phi}\right)+ \chi\phi^m-\frac{1}{m}=0, &&  \text{in} \, B_R, \\
    			&	\phi>0\quad\text{in} \, B_R, \quad	\phi=\frac{\partial \phi}{\partial n}=0, && \text{on}\, \partial B_R,
    		\end{aligned}\right.
    	\end{align}
where $B_R$ is an open ball in $R^N$. Such a solution can be smoothly extended by zero to yield a globally defined nonnegative smooth solution.

Once either of the above two types of solutions is found, we then have
$$
\psi(\xi)=(K*\phi^m)(\xi),
$$
since $\phi(\xi)=\phi(|\xi|)$, therefore,
\begin{align}\label{radial-1}
\psi(r)=\left\{
\begin{aligned}
&-\int_0^r r\phi^m(s)ds-\int_r^{+\infty}s\phi^m(s)ds, &&  N=1,
\\
&-\ln r\int_0^r s\phi^m(s) ds-\int_r^\infty s\ln(s) \phi^m(s)ds, && N=2,
\\
&\frac{1}{(N-2)r^{N-2}}\int_0^r s^{N-1}\phi^m(s) ds+\frac{1}{N-2}\int_r^\infty s\phi^m(s)ds, && N\ge 3.
\end{aligned}\right.
\end{align}
To search for such solutions, we proceed with the following transformation.

When $p\ne 2$, we let
\begin{align}\label{3-1}
u\left( r \right) =\phi ^{\frac{p-2}{p-1}}\left( r \right),
\end{align}
and
\begin{align}\label{3-2}
q: = \frac{m(p-1)}{p-2}.
\end{align}

$\bullet$ For the slow diffusion case $p>2$, the problem \eqref{2-7} is transformed into
	\begin{align}\label{3-3}
		\left\{
		\begin{aligned}
			&\left( B\left| u^{\prime}\left( r \right) \right|^{p-2}u^{\prime}\left( r \right) \right) ^{\prime}+\frac{B\left( N-1 \right)}{r}\left| u^{\prime}\left( r \right) \right|^{p-2}u^{\prime}\left( r \right) +\chi |u|^{q-1}u-\frac{1}{m}=0
			,\\
			&	u\left( 0 \right) =a, u^{\prime}\left( 0 \right) =0,\\
		\end{aligned}\right.
	\end{align}
where $B=| \frac{p-1}{p-2} | ^{p-1}$ and $a=A^{\frac{p-2}{p-1}}.$ Observing that $m=\frac{(p-2)N+p}{N}$, this indicates that $q>p-1$ when $p>2$.
It is straightforward to see that  $u^{*} \equiv \left( \frac{1}{\chi m} \right) ^{\frac{1}{q}}$ is the equilibrium point of the problem \eqref{3-3}.
For convenience in the subsequent proof, we allow the solution to be negative; hence, we replac $u^{q}$ with $ | u|^{q-1}u$.

We aim to find a ground state (i.e., a positive radial solution such that  $\displaystyle\lim_{r\to\infty}u(r)=0$), or
a compact-support radial solution satifying \eqref{eq2-6}. From \eqref{3-1}, we see that
\begin{equation}\label{eq2-11}
\phi'(R)=\frac{p-1}{p-2}\phi^{\frac1{p-1}}(R)u'(R)=\frac{p-1}{p-2}u^{\frac1{p-2}}(R)u'(R).
\end{equation}
Thus, when $u(R)=0$, it follows that $\phi'(R)=0$ if $u'(R)$ is finite.

The following proposition reveals that the problem \eqref{3-3} admits no ground state solution. Instead, only compactly supported solutions may exist, and their existence depends on the ranges of  of $p$ and $q$.

\begin{proposition}\label{pro-1}
Assume that $N >1$, $p > 2$, and $q > p - 1$. Then the following holds:
\begin{enumerate}
    \item[(i)] If $q\geq \frac{Np}{N - p} - 1$ with $2<p< N$, then for any given $a>0$, the problem \eqref{3-3} possesses a unique global positive solution $u(r) \in C^2(0, +\infty)$ such that
  $$\lim_{r \to \infty} u(r) = u^*.$$

    \item[(ii)] If either $q<\frac{Np}{N-p}-1$ with $2<p< N$, or $p\geq N$,
    then there exist constants $a_1$, $a_2$ with $a_2\geq a_1>u^*$, and $a_c\in [a_1, a_2]$ such that:
    \begin{itemize}
        \item For any $a\in (0, a_1)$, the problem \eqref{3-3} possesses a unique global positive solution $u(r) \in C^2(0, +\infty)$ with $\lim\limits_{r \to \infty} u(r)=u^*$.

        \item For $a = a_c$, the problem \eqref{3-3} admits a unique compact-support solution $u(r) \in C^2(0, R) \cap C^1[0, R]$ satisfying:
        $$
        u(r)>0 \ \text{ for }\ 0<r<R, \quad \text{and} \quad u(R)= u'(R)=0.
       $$

        \item For any $a\in (a_2, +\infty)$, the problem \eqref{3-3} admits a unique solution $u(r)\in C^2(0, R)\cap C^1[0, R]$ satisfying:
        $$
        u(r)>0 \ \text{ for }\ 0<r<R,\ \text{and} \ u(R) = 0,  u'(R)<0.
        $$
    \end{itemize}
    In particular, when $q<\frac{Np}{N-p}-1$ with $2<p< N$, it holds that $a_1=a_2=a_c$.
\end{enumerate}
\end{proposition}
In particular, when $N=1$, we have the following proposition.
 \begin{proposition}\label{pro-1-2}
Assume that $N=1$, $p > 2$, and $q > p - 1$. Then there exists a positive constant  $a_c=\left(\frac{q+1}{m\chi}\right)^{\frac1{q}}$ such that:
   \begin{itemize}
        \item For any $a\in (0, a_c)$, the problem \eqref{3-3} possesses a unique global positive oscilatory solution with constant amplitude around  the equilibrium point $u^*$.

        \item For $a = a_c$, the problem \eqref{3-3} admits a unique compact-support solution $u(r) \in C^2(0, R) \cap C^1[0, R]$ satisfying:
        $$
        u(r)>0 \ \text{ for }\ 0<r<R, \quad \text{and} \quad u(R)= u'(R)=0.
       $$

        \item For any $a\in (a_c, +\infty)$, the problem \eqref{3-3} admits a unique solution $u(r)\in C^2(0, R)\cap C^1[0, R]$ satisfying:
        $$
        u(r)>0 \ \text{ for }\ 0<r<R,\ \text{and} \ u(R) = 0,  u'(R)<0.
        $$
    \end{itemize}
   \end{proposition}
The above propositions suggest that

\begin{proposition}\label{pro-1-1}
When $p>2$. The problem \eqref{3-3} admits no positive radial ground state for any $a>0$; and the homogeneous Dirichlet-Neumann free boundary problem
  \begin{align}\label{3-10}
\left\{
    		\begin{aligned}
    			&B\Delta _pu+ \chi|u|^{q-1}u-\frac{1}{m}=0, &&  \text{in} \, B_R, \\
    			&	u>0\quad\text{in} \, B_R, \quad	u=\frac{\partial u}{\partial n}=0, && \text{on}\, \partial B_R
    		\end{aligned}\right.
    	\end{align}
admits a positive radial solution when $p-1<q< \frac{Np}{(N-p)_+}-1$.
\end{proposition}

\begin{remark}\label{re-1}
{\bf When $u(R)=0$, from \eqref{eq2-11}, it follows that $\phi'(R)=0$ since $u'(R)$ is finite}.
Consequently, when $p>2$ and $p-1<q< \frac{Np}{(N-p)_+}-1$, for any $a>a_2$ or $a=a_c$,
The solution of \eqref{2-7} satisfies the homogeneous Dirichlet-Neumann free boundary condition $\phi(R)=\phi'(R)=0$ for some $R>0$.
\end{remark}

$\bullet$ For the fast diffusion case $1<p<2$, \eqref{2-7} is transformed into
	\begin{align}
		\label{3-4}\left\{
		\begin{aligned}
			&\left( B\left| u^{\prime}\left( r \right) \right|^{p-2}u^{\prime}\left( r \right) \right) ^{\prime}+\frac{B\left( N-1 \right)}{r}\left| u^{\prime}\left( r \right) \right|^{p-2}u^{\prime}\left( r \right) -\chi |u|^{q-1}u+\frac{1}{m}=0
			,\\
			&	u\left( 0 \right) =a, u^{\prime}\left( 0 \right) =0,\\
		\end{aligned}\right.
	\end{align}
	where $q<0$, $B=| \frac{p-1}{p-2} |^{p-1}$ and $a=A^{\frac{p-2}{p-1}}$.

$\bullet$ For the linear diffuison case $p=2$, we let
\begin{align}\label{3-5}
u\left( r \right) =\ln \phi,
\end{align}
and \eqref{2-7} is transformed into
    	\begin{align}
    		\label{3-6}\left\{
    		\begin{aligned}
    			&u''\left( r \right) +\frac{N-1}{r}u^{\prime}\left( r \right) +\chi e^{mu}-\frac{1}{m}=0
    			,\\
    			&	u\left( 0 \right) = b,  u^{\prime}\left( 0 \right) =0,
    			\\
    		\end{aligned}\right.
    	\end{align}
    where $b = \ln A$. Clearly, $u_*= \frac{1}{m}\ln \frac{1}{\chi m}$ is a stationary solution of the problem \eqref{3-6}.

For the fast diffusion case $1<p\le 2$, it is straightforward to observe that $\phi\to 0$ is equivalent to $u\to \infty$. In fact, the problem \eqref{2-7} does not admit any solution for which  $\phi \to 0$ at a finite point or at infinity.

\begin{proposition}\label{pro-2}
	Assume that $N\geq1$, \eqref{2-2} holds.

 (i)If $1<p<2$, the problem \eqref{2-7}  admits neither a ground state solution nor a solution that vanishes at a finite point. In fact, if the problem \eqref{2-7} admits a global solution, then the solution must have a positive lower bound.

 (ii)If $p=2$, the problem \eqref{2-7} admits a global positive solution $\phi(r) \in C^1[0,\infty)$, which oscillates  around $(\frac{1}{\chi m})^{\frac{1}{m}}$.
 In particular,  $\displaystyle\lim_{r\to\infty}\phi(r)=(\frac{1}{\chi m})^{\frac{1}{m}}$ if $N>1$; and the amplitude is constant when $N=1$, which means that $\displaystyle\lim_{r\to\infty}\phi(r)$ does not exist.
\end{proposition}

From Propositions \ref{pro-1}, \ref{pro-2} and Remark \ref{re-1}, we observe that no ground state solution exists for any fast or slow diffusion case $p > 1$. However, in the slow diffusion regime where $p>2$ and  $q < \frac{Np}{(N-p)_+}-1$, there exist smooth solutions with compact support, satisfying homogeneous Dirichlet-Neumann free boundary condition.

According to the conditions $2<p<N$ with $p-1<q< \frac{Np}{N-p}-1$, and $N\leq p$, the admissible parameter ranges for $p$ (or equivalently for $m$) that guarantee the existence of such compactly supported solutions are given in Table 1.

\vspace{0.5em}

\noindent

\textbf{Table 1. Parameter ranges for the existence of compactly supported solutions.}
\begin{center}
\begin{tabular}{|c|c|c|}
						\hline
					$N$ & the range corresponding to $p$  \\
					\hline
					$1$ & $p>2$ \\
					\hline
				    $2$  &  $p>2$ \\
					\hline
					$N \geqslant 3$  & $p>\frac{\sqrt{5N^2+2N+1}+3N+1}{2(N+1)}$  \\
					\hline
					\end{tabular}
$\Leftrightarrow$
\begin{tabular}{|c|c|c|}
						\hline
					$N$& the range corresponding to $m$  \\
					\hline
					$1$  & $m>\frac{2}{N}$\\
					\hline
				    $2$ &$m>\frac{2}{N}$\\
					\hline
					$N \geqslant 3$  &$ m> \frac{\sqrt{5N^2+2N+1}-N+1}{2N}$ \\
					\hline
				\end{tabular}
\end{center}
Here $m$ and $p$ are related by $p=\frac{(m+2)N}{N+1}$  (i.e. $m=\frac{(p-2)N+p}{N}$ ).

From Proposition \ref{pro-1}, Proposition \ref{pro-2} and Remark \ref{re-1}, we can conclude that

$\bullet$ Under fast diffusion conditions $1<p\le 2$, no self-similar blow-up solutions with finite mass exist.

$\bullet$ Under slow diffusion conditions $p>2$, the system \eqref{1-5}  admits infinitely many self-similar blow-up solutions with compact support and finite mass, which is a weak solution of \eqref{1-5}, that is for any given $T_0<T$, for any $\varphi(x,t)\in C_0^\infty(\mathbb R^N\times (0, T_0))$,
\begin{align}\label{weaksolution}
\int_0^{T_0}\int_{\mathbb R^N}\rho _t\varphi dxdt+\int_0^{T_0}\int_{\mathbb R^N} | \nabla \rho|^{p-2}\nabla \rho\nabla\varphi dxdt=\chi\int_0^{T_0}\int_{\mathbb R^N} \rho \nabla c\nabla\varphi dxdt.
\end{align}
In fact, we note  that
$$
\rho \left( x,t \right) =( T-t ) ^{-\frac{1}{m}}\phi ( ( T-t ) ^{-\frac{1}{mN}}\left| x \right| ) ,\quad c( x,t) =( T-t)^{\frac{2-mN}{mN}}
\psi(( T-t)^{-\frac{1}{mN}}\left| x \right|),
$$
and $\rho\in C^1(\mathbb R^N\times (0, T_0))$, which satisfies the equation \eqref{1-5} in the classical sense when $(T-t)^{-\frac{1}{mN}}|x|<R$, and
\begin{equation}\label{boundary}
\rho=0, \ \text{for} \  |x|\ge R(t); \quad
\left.\frac{\partial\rho}{\partial{\bf n}}\right|_{\partial B_{R(t)}}=0,
\end{equation}
where $R(t)=(T-t)^{\frac{1}{mN}}R$. Therefore,
$$
\int_0^{T_0}\int_{B_{R(t)}}\rho _t\varphi dxdt-\int_0^{T_0}\int_{B_{R(t)}} \nabla\cdot(| \nabla \rho|^{p-2}\nabla \rho)\varphi dxdt=-\chi\int_0^{T_0}\int_{B_{R(t)}} \nabla\cdot(\rho \nabla c)\varphi dxdt.
$$
Using the boundary condition \eqref{boundary}, we arrive at
\begin{align*}
\int_0^{T_0}\int_{B_{R(t)}}\rho _t\varphi dxdt+\int_0^{T_0}\int_{B_{R(t)}} | \nabla \rho|^{p-2}\nabla \rho\nabla\varphi dxdt=\chi\int_0^{T_0}\int_{B_{R(t)}} \rho \nabla c\nabla\varphi dxdt.
\end{align*}
From  \eqref{boundary}, we further obtain \eqref{weaksolution}.

We have the following conclusion for slow diffusion case.

\begin{theorem}[{\bf Blow-up self-similar solution with shrinking compact support}]
	\label{th-3}
Assume that $m=\frac{(p-2)N+p}{N}$. Let $\left( \rho ,c \right)$ be a backward self-similar solution  with compact support of \eqref{1-5}  as constructed in Propositions \ref{pro-1}, \ref{pro-1-2}, i.e.,
	$$
	\rho \left( x,t \right) =( T-t ) ^{-\frac{1}{m}}\phi ( ( T-t ) ^{-\frac{1}{mN}}\left| x \right| ) ,\quad c\left( x,t \right) =\left( T-t \right) ^{\frac{2-mN}{mN}}\psi ( ( T-t ) ^{-\frac{1}{mN}}\left| x \right|),
	$$
where the parameters satisfy that (i) when $N=1, 2$, $p>2$; (ii) when $N\ge 3$, $p>\frac{\sqrt{5N^2+2N+1}+3N+1}{2(N+1)}$.
Then these self-similar solutions possess the following properties:

$\bullet$ {\bf Finite-time blow-up and concentration}. The solution  blows  up at time  $T$ , and as  $t \to T^{-}$  the bacterial density concentrates at the origin in the sense of distributions, that is
$$\rho \left( x,t \right) \rightarrow \mathcal{M} \left( A \right) \delta \left( x \right) , \, as \, \ t \rightarrow T^-,$$
where $\mathcal{M} \left( A\right) = |\partial B_1|\int_0^{R\left( A\right)}{\phi(r)r^{N-1}dr}$ and $\delta \left( x \right)$ is the Dirac $\delta$-function.

$\bullet$ {\bf Compact support and radial monotonicity}. The profile  $\phi$  is radially symmetric, non-increasing, and has compact support.

$\bullet$ {\bf Shrinkage of the support}. The compact support of the solution shrinks toward the origin as  $t \to T^{-}$ at the rate  $(T-t)^{\frac{1}{mN}}$.
\end{theorem}

\begin{remark}
This result shows that, under the stated parameter ranges, the  $p$-Laplacian diffusion gives rise to blow-up self-similar solutions that are compactly supported and exhibit support shrinkage, which contrasts sharply with the infinitely expanding supports known for linear diffusion.
\end{remark}

$\bullet$ {\bf Construction of blow-up solutions with multiple bubbles.}
 In fact, based on the previous results, we can construct blow-up solutions with any finite number of bubbles. For example, in the one-dimensional spatial case, when $a > a_c$, $u$ oscillates around $u^*$ with constant amplitude (as illustrated in Figure \ref{fig1}).
\begin{figure}[htbp]
	\centering
	\includegraphics[scale=0.4]{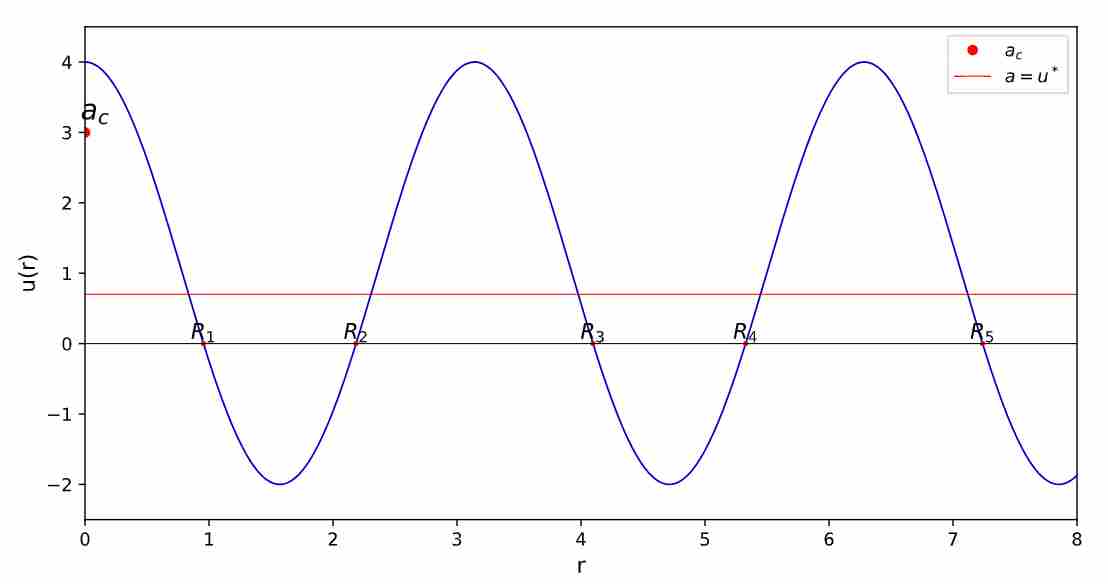}
	\caption{ Various oscillating behaviours of $u(\cdot , r)$ with $N=1$. }
	\label{fig1}
\end{figure}
\begin{figure}[htbp]
	\centering
	\includegraphics[scale=0.4]{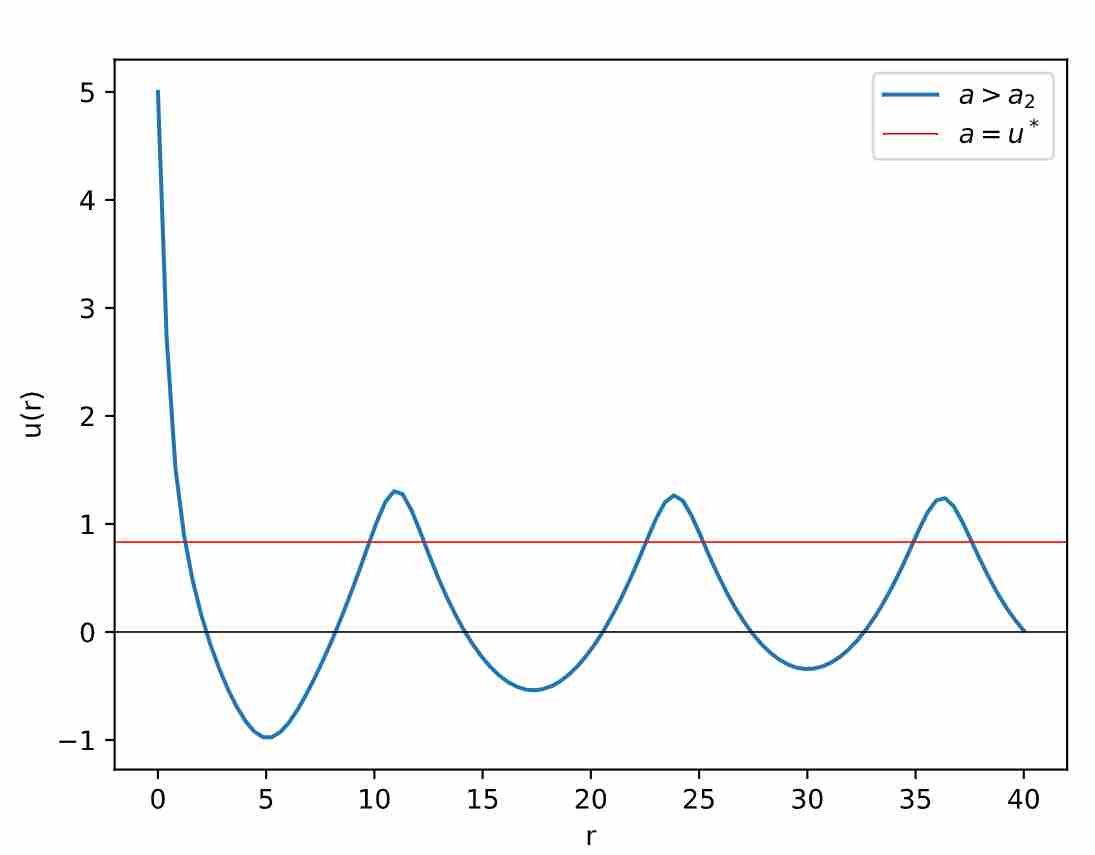}
	\caption{ Various oscillating behaviours of $u(\cdot , r)$ with $p\geq N>1$. }
	\label{fig2}
\end{figure}\\
Let the zeros of $ u$ be successively denoted by $\{ R_i \}_{i=1}^\infty$. We now describe, as an example, the construction of a blow-up solution with three bubbles. Define
$$
\phi(r) =
\begin{cases}
u^{\frac{p-1}{p-2}}(r), & \text{for } 0 < r < R_1, \; R_2 < r < R_3, \; R_4 < r < R_5, \\
0, & \text{otherwise.}
\end{cases}
$$
Then the support of $\rho$ consists of the following regions,
$$
{\bf supp} \ \rho(\cdot, t)=\{x; |x| < (T-t)^{\frac{1}{mN}}R_1,\
(T-t)^{\frac{1}{mN}}R_2 < |x| < (T-t)^{\frac{1}{mN}}R_3,\
(T-t)^{\frac{1}{mN}}R_4 < |x| < (T-t)^{\frac{1}{mN}}R_5\}
$$
and $\rho$ finally converges to Dirac $\delta$-function as $t\to T^-$. For higher dimensional cases (see Figure \ref{fig2}), such  blow-up solution
with multiple bubbles can also be constructed using similar methods.

\section{Analysis of Backward Self-Similar Blow-up Solutions}

 In this section, we search for two types of solutions to the transformed problem  \eqref{2-7}:

$\bullet$ ground state solutions-nonnegative, nontrivial, and continuously differentiable distribution solutions  satisfing $\lim_{r\to\infty}\phi(r)=0$;

$\bullet$ compact support solutions  that satisfy a homogeneous Dirichlet-Neumann free boundary condition, as illustrated by the problem \eqref{eq2-6}.
\medskip
\\
Because the transformed equation takes different forms for the slow diffusion ($p>2$), fast diffusion ($1<p<2$), and linear diffusion ($p=2$) cases, we study the associated initial value problems \eqref{3-3}, \eqref{3-4}, and \eqref{3-6}), respectively.

Firstly, the local existence of solutions to the above three problems can be readily established using standard fixed-point methods. For instance, the solution to \eqref{3-3} can be represented by the following integral equation
$$
u(r) = a+\int_0^r{\left( \frac{1}{B\tau^{N-1}} \right) ^{\frac{1}{p-1}}\left| \int_0^{\tau}{\left( \frac{1}{m}-\chi \left| u \right|^{q-1}u \right) s^{N-1}ds} \right|^{\frac{2-p}{p-1}}\int_0^{\tau}{\left( \frac{1}{m}-\chi \left| u \right|^{q-1}u \right) s^{N-1}ds}d\tau}.
$$
By defining the mapping
$$
Tu := a+\int_0^r{\left( \frac{1}{B\tau^{N-1}} \right) ^{\frac{1}{p-1}}\left| \int_0^{\tau}{\left( \frac{1}{m}-\chi \left| u \right|^{q-1}u \right)s^{N-1}ds} \right|^{\frac{2-p}{p-1}}\int_0^{\tau}{\left( \frac{1}{m}-\chi \left| u \right|^{q-1}u \right) s^{N-1}ds}d\tau} .
$$
and applying the fixed-point framework, local existence can be confirmed. The proof follows standard techniques but is lengthy, hence, we omit the details here and present only the conclusion.
	\begin{lemma}\label{le3.1}
		Assume that $a>0, p>2$, (or $p = 2$). Then the problem \eqref{3-3} (or  \eqref{3-6}) admits a unique solution $u(r) \in C^1[0, R_{max})$ such that either $R_{max}=\infty$, or  $$\underset{r\rightarrow R_{max}^{-}}{\lim}\,\,u(r)= \infty.$$
	\end{lemma}
\begin{remark}
	When $1<p<2$, $q<0$. Therefore, if $u$ have a positive lower bound, then the above local existence lemma also holds for $1<p<2$.
\end{remark}

\subsection{The slow diffusion case $p>2$}

{\bf Notations:} The solution of \eqref{3-3} depends continuously on the initial value $a>0$; to make this dependence explicit we write $u(r,a)$for the solution corresponding to $u(0)=a$.
\begin{lemma}\label{le3.2}
		Assume that $a>0, \, \chi >0, m>0, p>2, q>p-1, N\geq1$. Let $u(\cdot, a) \in C^1[0, R_{max}(a))$ be a classical solution of \eqref{3-3}, and by $(0, R_{max}(a))$ its maximal interval of existence, then $R_{max}(a)=\infty$, i.e. the solution exists globally in $r$.
	\end{lemma}

	\begin{proof}
We first show that $R_{\max}(a)=\infty$ for any $a>0$. By Lemma \ref{le3.1}, it suffices to establish the boundedness of $u$. To achieve this, we examine the energy functional associated with $u(r,a)$. Denote
		\begin{align}\label{3-8}
		E\left( r, a \right): =\frac{B\left( p-1 \right)}{p}\left| u^{\prime}\left( r,a \right) \right|^p+ \frac{\chi}{q+1}|u(r,a) |^{q+1}-\frac{1}{m}u(r,a),
		\end{align}
with
$$
E(0,a)=\frac{\chi}{q+1}a^{q+1}-\frac{a}{m}.
$$
Using \eqref{3-3}, a direct calculation yields
\begin{equation}\label{eq3-2}
\frac{dE}{dr}\left( r,a \right) =-\frac{B\left( N-1 \right)}{r}\left| u^{\prime}\left( r,a \right) \right|^p\leqslant 0.
\end{equation}
Obviously, we have
$$
E\left( r ,a\right) \geqslant  \frac{\chi}{q+1}|u |^{q+1}-\frac{1}{m}u.
$$
Denote
\begin{align}\label{3-9}
		g(u):=\frac{\chi}{q+1}|u |^{q+1}-\frac{1}{m}u.
		\end{align}
It is not difficulty to arrive at $g(u)$ reaches its minimum at $u^*$, that is
$$
g(u)_{min}=g(u^*)=-\frac{q}{q+1}u^{*}<0,
$$
and
$$
g(u)<0 , \ \text{for}\ u\in (0, u^*],
$$
where $u^*=\left( \frac{1}{\chi m} \right) ^{\frac{1}{q}}$ is the equilibrium point of the equation \eqref{3-3}.
Since $E(r,a)$ is non-increasing, then
$$
E\left( r,a \right) \in \left[-\frac{q}{q+1}u^{*}, E\left( 0,a \right) \right] ,$$ which implies that $u(r,a)$ is bounded. Therefore, $R_{max} =\infty$ .
\end{proof}

Next, we proceed to prove the conclusion of Propositions \ref{pro-1}, \ref{pro-1-2}. For the convenience of subsequent analysis, we collect the oscillation criteria for second-order differential equations with $p-$Laplacian form \cite{11}, which will be used in the proof of Lemma \ref{le3.4} and Lemma \ref{le3.5}.
	\begin{lemma}\label{le3.3}
		Consider the quasilinear differential equation
		\begin{align*}
			&\left( t^{N-1}\left| u^{\prime} \right|^{p\left( t \right) -2}u^{\prime} \right) ^{\prime}+t^{N-1}F\left( t \right) \left| u \right|^{\mu -2}u=0 \quad for\ t>t_0,
		\end{align*}
		where $p(t)>1$ and $\mu >1$. Suppose that  for any $L>0$,
		\begin{align*}
			&\int_{t_0}^{\infty}{\left( \frac{L}{t^{N-1}} \right) ^{\frac{1}{p\left( t \right) -1}}}dt=\infty
		\end{align*}
and $$\int_{t_0}^{\infty}{t^{N-1}F\left( t \right) dt}=\infty.$$
		Then, all radially symmetric solutions of the above equation are oscillatory.
	\end{lemma}
Using the lemma above, when $p\geq N$ , we obtain the following oscillation lemma.
\begin{lemma}\label{le3.4}
		Assume that $N\ge 1, p>2, q>p-1$. Let $u(\cdot, a) \in C^1[0, \infty)$ be a classical solution of \eqref{3-3}.
When $p\geq N$, $u\left(\cdot,a \right)$ oscillates around $u^*$ on $\left( 0,\infty \right)$, where $u^{*} \equiv \left( \frac{1}{\chi m} \right) ^{\frac{1}{q}}$ is the equilibrium point of the problem \eqref{3-3}.
To be more precisely, we have the following conclusions.

\medskip
\begin{enumerate}
\item[(i)] If $a>u^{*}$, there is an increasing sequence $\left( r_i\left( a \right) \right) _{i\geq 0}$ of real numbers with $r_0(a)=0$, such that:
 \begin{itemize}
\item when $N=1$,
		\begin{align*}
			\left\{
			\begin{aligned}
				&	u^{\prime}\left( r_i\left( a \right) ,a \right) =0,\quad  \left( -1 \right) ^iu^{\prime}\left( r,a \right) <0,\quad   r\in \left( r_i\left( a \right) , r_{i+1}\left( a \right) \right) ,\\
				&u\left( r_{2i}\left( a \right) ,a \right) =u\left( r_{2i+2}\left( a \right) ,a \right) >u^{*}>u\left( r_{2i+3}\left( a \right) ,a \right) =u\left( r_{2i+1}\left( a \right) ,a \right),
			\end{aligned}\right.
		\end{align*}
\item when $N\geq2$,
		\begin{align*}
			\left\{
			\begin{aligned}
				&	u^{\prime}\left( r_i\left( a \right) ,a \right) =0,\quad  \left( -1 \right) ^iu^{\prime}\left( r,a \right) <0,\quad   r\in \left( r_i\left( a \right) , r_{i+1}\left( a \right) \right) ,\\
				&u\left( r_{2i}\left( a \right) ,a \right) >u\left( r_{2i+2}\left( a \right) ,a \right) >u^{*}>u\left( r_{2i+3}\left( a \right) ,a \right) >u\left( r_{2i+1}\left( a \right) ,a \right),
			\end{aligned}\right.
		\end{align*}
		for $i\geq 0$.
\end{itemize}

\item[(ii)] If $0<a<u^{*}$, there is an increasing sequence $\left( r_i\left( a \right) \right) _{i\geq1}$ of real numbers with $r_1(a)=0$, such that:

 \begin{itemize}
  \item when $N=1$,
 \begin{align*}
			\left\{
			\begin{aligned}
				&u^{\prime}\left( r_i\left( a \right) ,a \right) =0, \quad  \left( -1 \right) ^iu^{\prime}\left( r,a \right) <0,\quad   r\in \left( r_i\left( a \right) , r_{i+1}\left( a \right) \right) ,\\
				&u\left( r_{2i}\left( a \right) ,a \right) =u\left( r_{2i+2}\left( a \right) ,a \right) >u^{*}>u\left( r_{2i+1}\left( a \right) ,a \right) =u\left( r_{2i-1}\left( a \right) ,a \right),
			\end{aligned}\right.
		\end{align*}
\item when $N\geq2$,
		\begin{align*}
			\left\{
			\begin{aligned}
				&u^{\prime}\left( r_i\left( a \right) ,a \right) =0, \quad  \left( -1 \right) ^iu^{\prime}\left( r,a \right) <0,\quad   r\in \left( r_i\left( a \right) , r_{i+1}\left( a \right) \right) ,\\
				&u\left( r_{2i}\left( a \right) ,a \right) >u\left( r_{2i+2}\left( a \right) ,a \right) >u^{*}>u\left( r_{2i+1}\left( a \right) ,a \right) >u\left( r_{2i-1}\left( a \right) ,a \right),
			\end{aligned}\right.
		\end{align*}
		for $i\geq 1$.
\end{itemize}
\end{enumerate}
	\end{lemma}
\begin{figure}[htbp]
	\centering
	\includegraphics[scale=0.4]{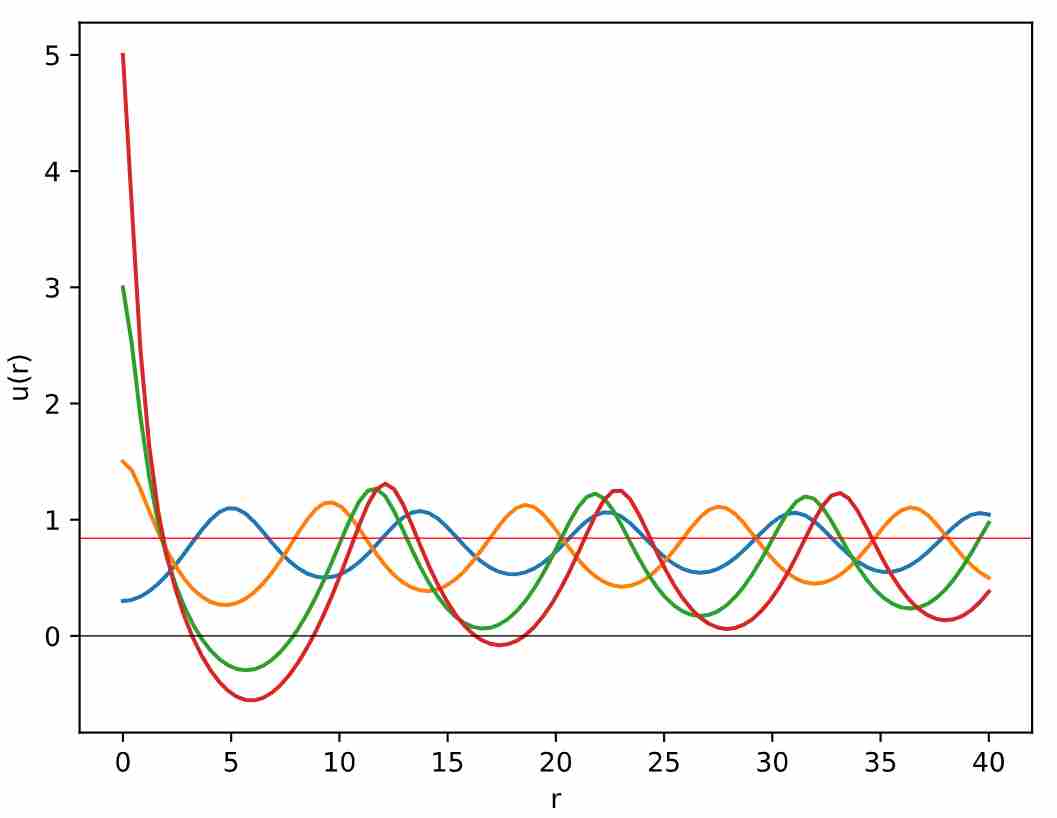}
	\caption{ Various oscillating behaviours of $u(\cdot , r)$ with $p\geq N>1$. }
	\label{fig1:img}
\end{figure}
\begin{proof}
Without loss of generality, we assume $a > u^*$ (the case  $a < u^*$  can be handled similarly).
By L'Hospital's rule and the equation \eqref{3-3}, we observe that
\begin{align}\label{3-11}
		\left( \left| u^{\prime}\left( 0,a \right) \right|^{p-2}u^{\prime}\left( 0,a \right) \right) ^{\prime}=\frac{1}{BN}\left( \frac{1}{m}-\chi a^{q} \right) <0.
	\end{align}		
		Hence $u\left( r,a \right)$ attains its local maximum at $r=0$. This means that $u\left( r,a \right)$ decreases monotonically for small $r$.

(i)We show that $u(r, a)$ cannot remain monotonic for all time. Assume for contradiction that $u(r, a)$ decreases monotonically for all time. Due to the boundedness of $u(r,a)$ and \eqref{3-3}, we deduce
	    $$
		\lim _{r\rightarrow \infty}\chi \left| u \right|^{q-1}u-\frac{1}{m}= 0.
		$$
That is,
\begin{align}\label{3-12}
		\lim _{r\rightarrow \infty}u\left( r,a \right) =u^{*}.
	\end{align}	
Noticing that \eqref{3-12} and
\begin{align}\label{3-13}
		|u^{\prime}|^{p-2}u^{\prime}(r,a)=\frac{1}{Br^{N-1}}\int_0^{r}{\left( \frac{1}{m}-\chi \left| u \right|^{q-1}u \right) s^{N-1}}ds<0,
	\end{align}
therefore, $u\left( r,a \right) \in \left(u^{*},a \right] ,$ for $ r\in \left[ 0,\infty \right)$.

Denote
\begin{align}\label{3-14}
		v=u-u^*>0.
	\end{align}
Substituting \eqref{3-14} into \eqref{3-3}, we have
		\begin{align}\label{3-15}
			&\left( \left| v^{\prime} \right|^{p-2}v^{\prime} \right) ^{\prime}+\frac{\left( N-1 \right)}{r}\left| v^{\prime} \right|^{p-2}v^{\prime}+\frac{\chi u^{q}-\frac{1}{m}}{\left( u-u^{*} \right) ^q}|v|^{q-1}v=0.
		\end{align}
Noticing that $q>p-1>1$ and $\frac{u}{u^{*}} >1$, then
		$$\frac{\chi u^{q}-\frac{1}{m}}{\left( u-u^{*} \right) ^q}=\chi\frac{(\frac{u}{u^*})^q-1}{\left( \frac{u}{u^*}-1 \right) ^q}
\geq\chi.$$
Since $p\geq N$,  by Lemma \ref{le3.3}, all radially symmetric solutions of \eqref{3-15} oscillates around $v=0$, which contradicts $v>0$.	

(ii) Let $r_1$ be the first minimum point, then from \eqref{3-13}, we see that $u\left( r_1 ,a\right)<u^{*}$  (otherwise, $u'(r_1, a)<0$). Similar to the proof in (i), $u\left( r,a \right)$ cannot remain increasing in $\left( r_1,\infty \right)$. By repeating the above process, we obtain two sequences of extreme points,
		\begin{align*}
		\text{maximum points}\quad \left\{ 0=r_0<r_2<\cdots \right\} ,
		\end{align*}
		and
		\begin{align*}
		\text{minimum points}	\quad \left\{ r_1<r_3<\cdots \right\} ,
		\end{align*}
with $r_i<r_{i+1}$. From \eqref{eq3-2}, and observing the monotonicity of $E\left( r,a \right)$, it is clear that for $N=1$, $E(r, a)\equiv E(0, a)$. Therefore,
\begin{align*}
			u\left( r_0 ,a\right) =u\left( r_2 ,a\right) =u\left( r_4 ,a\right) =\cdots >u^{*},
		\end{align*}
		and
		\begin{align*}
			&u\left( r_1,a \right) =u\left( r_3 ,a\right) =u\left( r_5 ,a\right)=\cdots <u^{*} .
		\end{align*}
For $N\geq2$
		\begin{align*}
			u\left( r_0 ,a\right) >u\left( r_2 ,a\right) >u\left( r_4 ,a\right) >\cdots >u^{*},
		\end{align*}
		and
		\begin{align*}
			&u\left( r_1,a \right) <u\left( r_3 ,a\right) <u\left( r_5 ,a\right)<\cdots <u^{*} .
		\end{align*}
		This lemma is proved.
\end{proof}

The above lemmas show that the solution of the problem \eqref{3-3} always exists globally but may change sign. Since our goal is to find a nonnegative solution to \eqref{3-3}, we now introduce the following definition. Let $u(\cdot, a)$ be the global solution of \eqref{3-3}. For any $a>0$, define
$$
R(a) \coloneqq \sup\left\{ R>0:u(0)=a\,\, \text{and}  \,\,u\left( r \right) >0 , \,r\in \left[ 0,R \right) \right\}.
$$
Due to the positivity of $a$ and the continuity of $u\left(\cdot,a \right)$, we have $R\left( a \right) >0$. We consider the sets
		\begin{align*}
			&\mathcal{P} :=\left\{ a>0:R\left( a \right) =\infty \right\} , \\
			&\mathcal{N} :=\left\{ a>0:R\left( a \right) <\infty \quad \text{and} \quad u^{\prime}\left( R\left( a \right) ,a \right) <0 \right\},\\
			&\mathcal{N} _0:=\left\{ a>0:R\left( a \right) <\infty \quad \text{and} \quad u^{\prime}\left( R\left( a \right) ,a \right) =0 \right\} .
		\end{align*}
It is evident that the three sets are mutually disjoint and $\mathcal{P} \bigcup{\mathcal{N} \bigcup{\mathcal{N} _0=\left( 0,\infty \right)}}$.  We first show that both $\mathcal{P}$ and $\mathcal{N}$ are nonempty and open in $(0, \infty)$. Finally, using the connectedness of $(0, \infty)$ and the fact that $\mathcal{P}$ and $\mathcal{N}$ are nonempty open sets, we conclude that $\mathcal{N}_0$ is nonempty.
\begin{lemma}\label{le3.5}
		Assume that $N\ge 1$, $p>2$, $p-1<q<\frac{Np}{(N-p)_+}-1$. Then $\left( 0,u^{*} \right] \subset \mathcal{P}$ and there exists $a_2>u^*$ such that $\left( a_2,\infty \right) \subset \mathcal{N}$.
	\end{lemma}
\begin{proof}
First, we show that $\left( 0,u^{*} \right] \subset \mathcal{P}$. Obviously, $u^*\in \mathcal{P}$. For any $a\in \left( 0,u^{*} \right)$, we observe that $E(0, a)<0$. By the monotonicity of $E\left( r,a \right)$, we derive that $E\left( r,a \right) \leq E(0, a)<0$, for $r\in (0, \infty)$. We claim that $R(a) =\infty$. Otherwise, $E\left( R\left( a \right) ,a \right) =\frac{B\left( p-1 \right)}{p}\left| u^{\prime}\left( r,a \right) \right|^p\geq 0$. It is a contradiction. Therefore, $R(a) =\infty$, and we conclude that
$\left( 0,u^{*} \right] \subset \mathcal{P}$.

Next, we claim that there exists $a_2>u^*$ such that $\left( a_2,\infty \right) \subset \mathcal{N}$.
Assume for contradiction that there exists a sequence $\{a_j\}$ with $a_j\rightarrow \infty$ as $j\rightarrow \infty$ such that $a_j \notin \mathcal{N}$ for sufficiently large $j$. Let
\begin{align}\label{3-16}
			&\tilde{u_j}( r,a_j) =\frac{1}{a_j}u_j( \frac{r}{a_j^{\lambda}},a_j ) ,
		\end{align}
		where $
		\lambda =\frac{\left( p-1 \right) \left( m+2-p \right)}{p\left( p-2 \right)}>0.
		$ Substituting \eqref{3-16} into \eqref{3-3}, \eqref{3-3} is transformed into
		\begin{align}\label{3-17}
			\left\{
			\begin{aligned}
				&( B| \tilde{u_j}^{\prime}|^{p-2}\tilde{u_j}^{\prime}\ ) ^{\prime}+\frac{B\left( N-1 \right)}{r}| \tilde{u_j}^{\prime} |^{p-2}\tilde{u_j}^{\prime} +\chi | \tilde{u_j} |^{q-1}\tilde{u_j}-\frac{1}{m}{a_j}^{-q}=0,\\
				&\tilde{u_j}( 0,a_j ) =1, \tilde{u_j}^{\prime}( 0,a_j ) =0.
			\end{aligned}\right.
		\end{align}
By definition of $R( a_j )$, it is easy to know $\tilde{u_j}( r,a_j) >0$ for  $ r\in [ 0,{a}^{\lambda}_jR( a_j ))$.

In what follows, we show that for sufficiently large $a_j$, the sequence $\tilde{u_j}$ converges to $w$, where $w$ is the unique solution to the initial value problem
\begin{align}\label{3-18}
			\left\{
			\begin{aligned}
				&( B\left| w^{\prime}\left( r \right) \right|^{p-2}w^{\prime}\left( r \right) ) ^{\prime}+\frac{B\left( N-1 \right)}{r}\left| w^{\prime}\left( r \right) \right|^{p-2}w^{\prime}\left( r \right) +\chi \left| w\left( r \right) \right|^{q-1}w\left( r \right) =0,\\
				&	w\left( 0 \right) =1, w^{\prime}\left( 0 \right) =0.
			\end{aligned}\right.
		\end{align}

To establish the uniform boundedness of $\tilde{u_j}$, it suffices to prove that $u_j$ is uniformly bounded on $(0,R(a_j))$. We claim that $u_j(r)$ attains the maximum value at $r=0$. Suppose that contrary holds, then there exists a point $r^* \in (0,\infty)$ such that $u_j(r^*)>a_j>u^*$, with $u'_j(r^*)=0$. Using \eqref{3-9}, we have
 $$g_j(u_j(r^*))> g_j(a_j).$$
From \eqref{3-8}, this implies
$$E_j(r^*,a_j)>E_j(0,a_j),$$
contradicting the monotonicity of $E_j(r,a_j)$. Therefore,  we conclude that $$ u_j(r, a_j) \leq a_j$$
 for all $r \in [0,\infty)$. It is straightforward to conclude that $0 \leq u_j(r) \leq a_j$ for all $r \in [0, R(a_j)]$. From \eqref{3-16}, we observe that
		\begin{align}\label{3-19}
			&0\leq \tilde{u_j}\leq 1 ,
		\end{align}
		for $ r\in [0,a_j^{\lambda}R( a_j )]$.
		From \eqref{3-17} and \eqref{3-19}, we see that for $r \in [0, a_j^{\lambda}R( a_j )]$,
		\begin{align}\label{3-20}
			&B( | \tilde{u_j}^{\prime} |^{p-2}\tilde{u_j}^{\prime} ) =-\frac{1}{r^{N-1}}\int_0^r{\chi |  \tilde{u_j} |^{q-1} \tilde{u_j}s^{N-1}}ds+\frac{1}{r^{N-1}}\int_0^r{\frac{1}{m} {a_j}^{-q}s^{N-1}}ds \notag \\
			&\qquad\qquad \quad \leq \frac{1}{r^{N-1}}\int_0^r{\frac{1}{m}{a_j}^{-q}}s^{N-1}ds\leq \frac{1}{m}{a_j}^{-q}r
		\end{align}
		and
		\begin{align}\label{3-21}
			&B(| \tilde{u_j}^{\prime} |^{p-2}\tilde{u_j}^{\prime}) \geq -\frac{1}{r^{N-1}}\int_0^r{\chi |  \tilde{u_j} |^{q-1}\tilde{u_j}s^{N-1}}ds\geq -\chi r.
		\end{align}
By \eqref{3-19}, \eqref{3-20} and \eqref{3-21}, we see that $| \tilde{u_j} |^{q-1}\tilde{u_j}$ and $ | \tilde{u_j}^{\prime} |^{p-2}\tilde{u_j}^{\prime}$ are uniformly bounded on $\left[ 0, R \right]$ for any $0<R \leq {a}^{\lambda}_jR( a_j )$. Then by Arzel$\grave{a}$-Ascoli theorem, as $a_j\rightarrow \infty$,
		\begin{align}\label{3-22}
			&| \tilde{u_j}|^{q-1}\tilde{u_j} \rightarrow \left| w \right|^{q-1}w , \,
			\frac{1}{m} {a_j}^{-q}\rightarrow 0
		\end{align}
uniformly on $\left[ 0, R \right]$ for any  $0<R\leq a^{\lambda}R( a_j )$.		
By \eqref{3-17}, we see that
\begin{align*}		|\tilde{u_j}^{\prime}|^{p-2}\tilde{u_j}^{\prime}(r)=\frac{1}{Br^{N-1}}\int_0^{r}{( \frac{1}{m}-\chi | \tilde{u_j} |^{q-1}\tilde{u_j}) s^{N-1}}ds.
	\end{align*}
Then by \eqref{3-20} to \eqref{3-22}, we derive that
\begin{align}\label{3-23}
		| \tilde{u_j}^{\prime} |^{p-2}\tilde{u_j}^{\prime} \rightarrow\left| w^{\prime} \right|^{p-2}w^{\prime},
	\end{align}
as $a_j\rightarrow \infty$.
        By \eqref{3-17}, \eqref{3-22} and \eqref{3-23}, we see that $( | \tilde{u_j}^{\prime} |^{p-2}\tilde{u_j}^{\prime} ) ^{\prime} \rightarrow(\left| w^{\prime} \right|^{p-2}w^{\prime} ) ^{\prime}$ as $a_j\rightarrow \infty$.
		Thus, for any  $0<R \leq a^{\lambda}_jR( a_j )$, we conclude
		\begin{align}\label{3-24}
			\begin{aligned}
				&\lim _{a_j\rightarrow \infty}\underset{r\in \left[ 0, R\right]}{sup}| \tilde{u_j}( r,a_j ) -w\left( r \right) |=0.
			\end{aligned}
		\end{align}	
Based on Lemma \ref{le3.3} and the results from Reference \cite{17}, when $p\geq N$ or $2<p<N$ with $p-1<q<\frac{Np}{N-p}-1$, the solution $w\left( r \right)$  oscillates around $0$. Therefore, let $z_1>0$ be the first zero, then
\begin{align}\label{3-25}
			\begin{aligned}
				&w(r)>0, \,\, \text{and} \,\, w^{\prime}\left( r \right) <0 \,\, \text{for}\, \,r \in (0,z_1),  \,\, w\left( z_1 \right) =0,  \,\, w^{\prime}\left( z_1 \right) <0.
			\end{aligned}
		\end{align}	
By \eqref{3-25}, there exists $\delta >0$ such that $w\left( r \right) <0$ for all $r\in \left( z_1,z_1+\delta \right)$. From \eqref{3-24}, for any $r\in \left( z_1,z_1+\delta \right)$, we have $\tilde{u_j}\left( r \right) <0$ when $a_j$ is sufficiently large (depending on $r$). This implies that
\begin{align}\label{3-26}
		a^{\lambda}_jR( a_j ) \leq z_1
	\end{align}
for $a_j$ large enough.
If $\sigma \in \left( 0,z_1 \right)$, then $w\left( r \right) \geq w\left( \sigma \right) >0$ for $r\in \left[ 0,\sigma \right]$. From \eqref{3-24}, we infer that
 $\tilde{u_j}\left( r\right) >\frac{w\left( \sigma \right)}{2}>0$ for  all $r\in \left[ 0,\sigma \right]$, when $a_j$ is sufficiently large. Consequently, we have
		$\sigma <a^{\lambda}_jR( a_j )$ for $a_j$ large enough which implies
\begin{align}\label{3-27}
		z_1 \leq a^{\lambda}_jR( a_j ) .
	\end{align}
Combining \eqref{3-26} and \eqref{3-27}, we conclude that $\underset{a_j\rightarrow \infty}{\lim}\,\,a^{\lambda}_jR( a_j ) = z_1.$
Since $\lambda>0$, it is easy to see that $R( a_j )<\infty$ as $a_j\rightarrow \infty$. From \eqref{3-16}, \eqref{3-24} and \eqref{3-25}, we see that for  sufficiently large $a_j$,
		$$u_j(r)>0, \,\, \text{and} \,\, u_j^{\prime}\left( r \right) <0 \,\, \text{for}\, \,r \in (0,R(a_j)),  \,\, u_j( R(a_j) ) =0,\,\,u_j^{\prime}( R(a_j) ) <0,$$
which implies $a_j \in \mathcal{N}$ as $j\rightarrow \infty$. It is a contradiction.
\end{proof}
To prove that $\mathcal{N}$ is an open set, we need to establish the monotonicity of the solution on $(0,R(a))$.
\begin{lemma}\label{le3.6}
		Assume that   $N\ge 1$, $p>2$, $p-1<q<\frac{Np}{(N-p)_+}-1$.  Then for $a \in \mathcal{N} \bigcup{\mathcal{N} _0}$,  $u'(r,a)<0$ for all $r \in (0,R(a))$.
	\end{lemma}
\begin{proof}
By Lemma \ref{le3.5}, we see that $a>u^*$ if $a \in \mathcal{N} \bigcup{\mathcal{N} _0}$. Then by \eqref{3-11}, $u\left( r,a \right)$ attains its local maximum at $r=0$. That is,  $u\left( r,a \right)$ decreases monotonically for small $r$. We now claim that $u'(r,a)<0$ for all $r \in (0,R(a))$. Suppose, for contradiction, that there exists $r^*\in (0,R(a))$ such that $u'(r^*,a)=0$. By \eqref{3-13}, we obtain that $0<u(r^*,a)<u^*$. Due to the monotonicity of $E(r,a)$ and \eqref{3-8}, we see that
     $$0>E(r^*,a)\geq E(R(a),a)\geq 0,$$
     which is a contradiction.
\end{proof}
\begin{lemma}\label{le3.7}
		Assume that $N\ge 1$, $p>2$, $p-1<q<\frac{Np}{(N-p)_+}-1$. Then $\mathcal{P}$ and $\mathcal{N}$ are open subsets of $\left( 0,\infty \right)$.
	\end{lemma}
\begin{proof}
        First, we consider $a\in \mathcal{N}$. By the definition of $\mathcal{N}$ and Lemma \ref{le3.6}, for a given $a$, there exists $\varrho >R\left( a \right)$ and $\varepsilon >0$ depending on $a$ such that for all $r\in \left( 0,\varrho \right)$, we have
        $$u\left( \varrho ,a \right) <0 \quad \text{and} \quad u^{\prime}\left( r,a \right) <-2\varepsilon .$$
		By continuity, there exists $\delta \in \left( 0,a \right)$ such that for every $b\in \left( a-\delta ,a+\delta \right)$ and all $r\in \left( 0,\varrho \right)$, we have that
		$$
		u\left( \varrho ,b \right) <0 \quad \text{and} \quad u^{\prime}\left( r,b \right) <-\varepsilon.
		$$
		Since $u\left( 0,b \right) =b>0$ and $u\left( \varrho ,b \right) <0$,  we directly deduce that for each  $b\in \left( a-\delta ,a+\delta \right)$, we have
$$R\left( b \right) \in \left( 0,\varrho \right)\quad \text{with}\quad  u^{\prime}\left( R\left( b \right) ,b \right) <-\varepsilon <0.$$
		Therefore, $\left( a-\delta ,a+\delta \right)\subset \mathcal{N}$. That is, $\mathcal{N}$  is an open set in $\left( 0,\infty \right).$

       Next, we consider $a\in \mathcal{P}$. Here, we restrict our attention to the case $a > u^*$, and show that every such $a$ is an interior point of $\mathcal{P}$.

       (i) When $2<p<N$  with $q<\frac{Np}{N-p}-1$, if $u(r)$ is  monotonically decreasing and bounded, then $u(r,a)> u^*$ for $r \in [0, \infty)$ from \eqref{3-3}.
       For \eqref{3-13} and a given $a$,  there exists $\varepsilon>0$ depending on $a$ such that
       \begin{align}\label{3-28}
				&u'\left( r,a \right) <-2\varepsilon, \,\,\text{for\, all}\, \,r\in ( 0, \infty).
			\end{align}
       We now claim that $\mathcal{P}$ is an open set in $\left( 0,\infty \right)$. That is, there exists $\delta >0$ such that for each $b\in \left( a-\delta ,a+\delta \right)$, we have $R(b)=\infty$.
       Assume for contradiction that for any $\delta >0$, there exists some $b\in \left( a-\delta ,a+\delta \right)$ with $R(b)<\infty$. By continuity and \eqref{3-28}, we obtain that
       $$u'\left( r,b \right) <-\varepsilon,\,\, \text{for}\,\, r\in (0, R(b)],$$ which implies $b\in \mathcal{N}$.
       Since $\mathcal{N}$ is open in $\left( 0,\infty \right)$, then there exist $\delta_1 >0$ such that $\left( b-\delta_1 , b+\delta_1 \right)\subset \mathcal{N}$. Taking $\delta=\frac{\delta_1}{4}$, we have $a \in \left( b-\delta_1 , b+\delta_1 \right)$, meaning $a \in \mathcal{N}$. This contradicts the assumption that $a\in \mathcal{P}$.

       If $u(r)$ is non-monotonic, then let $r_1$ be the first minimum point, which satisfies $u'\left( r_1\left( a \right) ,a \right)=0$ and $u\left( r_1\left( a \right) ,a \right) \in (0, u^*)$ (otherwise, $u'\left( r_1\left( a \right), a\right)<0$ for \eqref{3-13}). Then for all $r\in \left[ 0,r_1\left( a \right) \right]$, we have
       $$u\left( r,a \right) \geq u\left( r_1\left( a \right) ,a \right) \,\,\text{and}\,\, E(r_1( a ),a) <0.$$
       Depending on continuity, there exists $\delta \in (0,a)$ such that for every $b\in \left( a-\delta ,a+\delta \right)$ and all $r \in (0,r_1(a))$, we know
        $$u\left( r ,b \right) \geq \frac{u( r_1( a ) ,a)}{2}, \quad
        u\left( r_1(a) ,b \right) \in (0,u^*), \quad  \text{and} \quad E(r_1(a),b)<0.$$
        Suppose, for contradiction, that there exists $ b\in \left( a-\delta ,a+\delta \right)$ such that $R\left( b \right) <\infty.$ From the definition and monotonicity of $E\left( r,b \right)$, we have
		$$
		0>E\left( r_1\left( a \right) ,b \right)\geq E\left( R\left( b \right) ,b \right) \geq 0,
		$$
		which is a contradiction. Therefore, it implies $\left( a-\delta ,a+\delta \right) \subset \mathcal{P}$.

       (ii)When $p\geq N$, Lemma \ref{le3.4} implies that $u\left( r,a \right) \geq u\left( r_1\left( a \right) ,a \right) \in \left( 0,u^{*} \right)$ for $r\in \left[ 0,r_1\left( a \right) \right]$ and $E(r_1( a ),a) <0$. Following an argument analogous to the non-monotonic case in (i), we readily conclude that $\mathcal{P}$ is an open set in $\left( 0,\infty \right)$.
\end{proof}
Combining the Lemma \ref{le3.4} to Lemma \ref{le3.7}, we obtain the following conclusion and prove that $\mathcal{N}_0$ is nonempty.
\begin{lemma}\label{le3.8}
		Assume that $N\ge 1$, $p>2$, $p-1<q<\frac{Np}{(N-p)_+}-1$. Then for the initial value problem \eqref{3-3}, there exist $a_2\geq a_1>u^{*}$ such that
\begin{itemize}
        \item when $0<a<a_1$, then $u\left( r,a \right) >0,\ for \ r\in \left[ 0,\infty \right)$. In particular, $\lim\limits_{r \to \infty} u(r)=u^*$ if $N\ge 2$; $\lim\limits_{r \to \infty} u(r)$ does not exist if $N=1$.

          \item when $a_1\leq a \leq a_2$, then there exists at least one point $a_c \in [a_1,a_2]$  such that there is $R\left( a_c \right) >0$ such that
		\begin{align*}
			\left\{
			\begin{aligned}
				&	u\left( R\left( a_c \right) ,a_c \right) =0,\\
				&u^{\prime}\left( R\left( a_c \right) ,a_c \right) =0, \\
				&u\left( r,a_c \right) >0,\,\, u'\left( r,a_c \right)<0, \quad r\in \left[ 0,R\left( a_c \right) \right);
			\end{aligned}\right.
		\end{align*}

\item if $a>a_2$, then there is $R\left( a \right) >0$ such that
		\begin{align*}
			\left\{
			\begin{aligned}
				&u\left( R\left( a\right) ,a\right) =0,\\
				&u^{\prime}\left( R\left( a \right) ,a \right) <0, \\
				&u\left( r,a \right) >0,\,\, u'\left( r,a \right)<0, \quad r\in \left[ 0,R\left( a \right) \right) .
			\end{aligned}\right.
		\end{align*}
  \end{itemize}
Further, when $2<p<N$ with $q<\frac{Np}{N-p}-1$, $a_c$ is unique such that $a_c=a_1=a_2$. When $N=1$, $a_c$ is unique and $a_c=a_1=a_2=\left(\frac{q+1}{m\chi}\right)^{\frac1{q}}$	
	\end{lemma}
	\begin{figure}[htbp]
		\centering
		\includegraphics[scale=0.4]{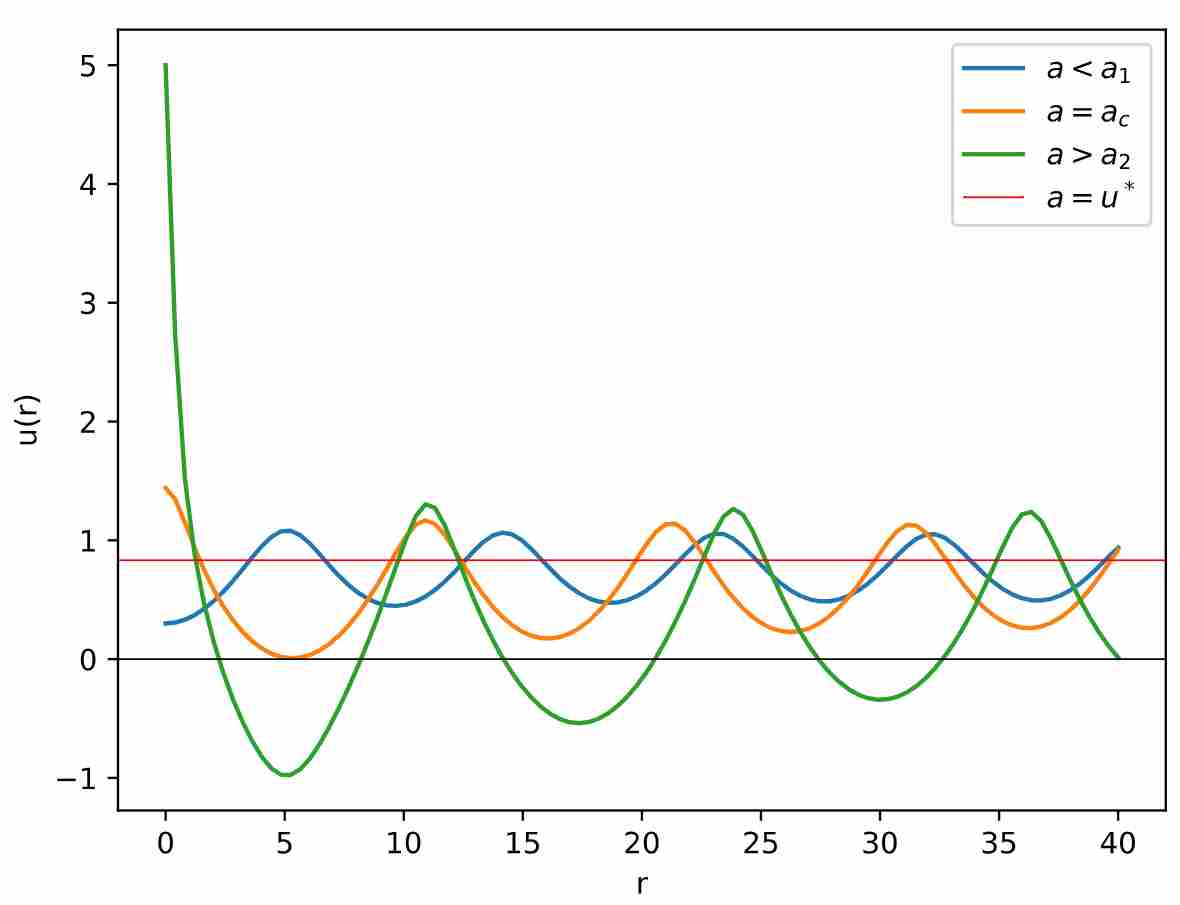}
		\caption{ Behaviour of $u(r, a )$ for $a<a_1$, $a=a_c$ and $a>a_2$ with $p\geq N>1$. }
		\label{fig2:img}
	\end{figure}	
\begin{proof}
First, we claim that $\mathcal{N} _0 \ne \emptyset$. Assume for a contradiction that $\mathcal{N} _0 =\emptyset$, then
$$\mathcal{P}\cup \mathcal{N} =(0,\infty) \quad \text{and} \quad \mathcal{P}\cap \mathcal{N} =\emptyset.$$ Since $(0,\infty)$ is a connected topological space, then it cannot be partitioned into two non-empty disjoint open sets, which is contradict with $\mathcal{P}\cup \mathcal{N} =(0,\infty)$.

When $N=1$, the proof of Lemma \ref{le3.2} shows that $u$ oscillates. From \eqref{eq3-2}, we see that the energy $E(r,a)$ is  conserved, which leads to a constant oscillation amplitude, so that $a_c$ is uniquely determined by $a_c=\left(\frac{q+1}{m\chi}\right)^{\frac1{q}}$, and if
$a<a_c$, the limit of $u$ does not exist as  $r\to\infty$; if $a>a_c$, the solution $u$ will reach $0$ at a finite position $R(a)$ such that $u^{\prime}\left( R\left( a \right) ,a \right) <0$.

When $N\geq 2$, since $u$ is bounded, then $$\underset{r\rightarrow \infty}{\lim}\,\,u \left( r \right) =u^* .$$
In fact, if $u$ is eventually monotone, i.e., if there exists $R > 0$ such that $u$ is monotone for $r > R$, then from equation \eqref{3-3}, it is clear that
$u(r)\to u^*$, as $r\to\infty$.
Otherwise, if $u$ oscillates, the proof of Lemma \ref{le3.2}
shows that when the decreasing energy $E(r,a)$ implies a decay in the oscillation amplitude, ultimately leading to convergence to $u^*$.

 From Lemma \ref{le3.5} to Lemma \ref{le3.7}, we obtain that there exists $a_2\geq a_1>u^*$ such that $(a_2, \infty) \subset \mathcal{N}$ and $u'(r,a)<0$ for all $r\in (0,R(a))$. Since $\mathcal{N} _0 \ne \emptyset$, then there exists at least one $a_c\in [a_1, a_2]$ such that $a_c \in \mathcal{N} _0$. By Lemma \ref{le3.6}, for $a_c \in \mathcal{N} _0$, we have $u'(r,a_c)<0$ for all $r\in (0,R(a_c))$.

From \cite{39}, we know that when $2<p<N$ with $q<\frac{Np}{N-p}-1$, \eqref{3-11} has a unique radial solution, which implies that $a_c$ is unique. Then $\mathcal{N} _0 = \{a_c\}$.  Since
that $\mathcal{P}$ and $\mathcal{N}$ are open sets of $(0, +\infty)$ with $\left( 0,a_1 \right) \subset\mathcal{P}$ and $\left( a_2,\infty \right) \subset \mathcal{N}$ $\left( 0,\infty \right)$, we quickly derive that $\mathcal{P} =\left( 0,a_c \right)$ and $\mathcal{N} =\left( a_c,\infty \right)$.

Finally, we conclude that when $2<p<N$ with $q<\frac{Np}{N-p}-1$, or $N=1$, we have $$\mathcal{P} =\left( 0,a_c \right),\,\,\mathcal{N} _0=\left\{ a_c \right\}, \,\, \mathcal{N} =\left( a_c,\infty \right).$$
	\end{proof}

\textbf{ \textit{Proof of Proposition \ref{pro-1}.}} From Lemma \ref{le3.8}, we see that \eqref{3-3} has no positive radial ground state.

(i) When  $q<\frac{Np}{N-p}-1$ with $2<p<N$, or $N=1$, $\mathcal{P} =\left( 0,a_c \right)$, $\mathcal{N}_0=\{a_c\}$, and $\mathcal{N} =\left( a_c,\infty \right)$.

(ii) When $p\ge N$, there exists $a_2\ge a_c\ge a_1>u^*$ such that $(0, a_1)\subset \mathcal{P}$, $(a_2, \infty) \subset \mathcal{N}$,  $a_c \in \mathcal{N} _0$.

(iii)When $q \ge \frac{Np}{N-p}-1$ with $2<p<N$, from the result of \cite{39}, \eqref{3-3} has no positive radial ground state and \eqref{3-11} has no radial solution. Following an argument completely analogous to the proofs of Propositions 2.5 and 2.6 in \cite{39}, we conclude that $u$ cannot reach zero. Thus, $\mathcal P=(0, +\infty)$ since $u$ is bounded. Additionally, similar to the proof in Lemma \ref{le3.8}, we obtain the asymptotic behavior
$$\lim_{r\to+\infty}u(r)=u^*.$$
The proof of Proposition \ref{pro-1} is complete. \hfill $\Box$

\subsection{The fast diffusion case $1<p\leq2$}	

	First, we consider the case $1<p<2$. Our objective is to find a solution such that $\underset{r\rightarrow R}{\lim}\,\,\phi \left( r \right) =0$ for some $R>0$, or $\underset{r\rightarrow \infty}{\lim}\,\,\phi \left( r \right) =0$. Since $\frac{p-2}{p-1}<0$, this is equivalent to finding a solution to problem \eqref{3-4} that satisfies  $\underset{r\rightarrow R}{\lim}\,\,u \left( r \right) =+\infty $, or $\underset{r\rightarrow \infty}{\lim}\,\,u \left( r \right) =+\infty $ from \eqref{3-1}. We will show that there is no such solution.

	\begin{lemma}\label{le3.9}
		Assume that $N\geq1, 1<p<2, q<0$. Let $u(r)$ be a classical solution of \eqref{3-4}, and $\left( 0,R_{max}(a) \right)$ be the maximal existence interval of the solution. Then $u(r)$ has upper bound on $\left( 0,R_{max}(a) \right)$.
		\begin{proof}
To derive an upper bound for $u$, we consider the corresponding energy functional for any $r\in \left[ 0,R_{\max}(a) \right)$.
When $p=\frac{2N+1}{N+1}\in (1,2)$, we observe that $q=-1.$
Therefore, we construct the energy functional in two cases.			
When $q\ne-1$, denote
			\begin{align}\label{3-29}
				&E\left( r, a \right): =\frac{B\left( p-1 \right)}{p}\left| u^{\prime}\left( r,a \right) \right|^p- \frac{\chi}{q+1}|u\left( r,a \right) |^{q+1}+\frac{1}{m}u\left( r,a \right).
			\end{align}
			While $q=-1$,  denote
			\begin{align}\label{3-30}
				&E\left( r, a \right): =\frac{B\left( p-1 \right)}{p}\left| u^{\prime}\left( r ,a\right) \right|^p-\chi \ln |u\left( r,a \right)|+\frac{1}{m}u\left( r,a \right).
			\end{align}
A direct calculation by \eqref{3-4} and \eqref{3-29}(or, \eqref{3-30}) yields
\begin{align*}
				&\frac{dE}{dr}(r, a) =-\frac{B\left( N-1 \right)}{r}\left| u^{\prime}\left( r ,a\right) \right|^p\leq 0,
\end{align*}
then $E(r,a)$ is decreasing on $\left( 0,R_{max}(a) \right)$. That is, for any $r \in \left( 0,R_{max}(a) \right)$, we have
\begin{align}\label{3-31}
				&E(r,a)\leq E(0,a).
\end{align}
By \eqref{3-29}, \eqref{3-30} and \eqref{3-31}, it is not difficult to infer that $u(r)$ has upper bound on $\left( 0,R_{max} (a)\right)$ since  $q<0$.
		\end{proof}
	\end{lemma}
	 Next, we consider the case $p=2$. We expect to find a global nonnegative solution of the problem \eqref{2-7} such that $\underset{r\rightarrow \infty}{\lim}\,\,\phi \left( r \right) =0$ or $\underset{r\rightarrow R_{max}}{\lim}\,\,\phi \left( r \right) =0$. By \eqref{3-5}, this is equivalent to seek a global solution of \eqref{3-6} such that $\underset{r\rightarrow \infty}{\lim}\,\,u\left( r \right) =-\infty$ or $\underset{r\rightarrow R_{max}}{\lim}\,\,u\left( r \right) =-\infty$ . We now prove that no such solution exists. That is,
\begin{lemma}\label{le3.10}
		Assume that $A>0, \chi>0, m>0, N\geq1, p=2$. Let $u\left( r \right)$ be a classical solution of \eqref{3-6}, and $\left( 0, R_{max} (b)\right)$ be the maximal existence interval of the solution. Then $R_{max}(b)=\infty$ and  $u(r)$ oscillates around $u_*$ on $(0, \infty)$, such that  when $N=1$, $\lim\limits_{r \to \infty} u(r)$ does not exist; when $N\geq 2$, 	$\underset{r\rightarrow \infty}{\lim}\,\,u \left( r \right) =u_*,$
where $u_*= \frac{1}{m}\ln \frac{1}{\chi m}$.
		\end{lemma}
		\begin{figure}[htbp]
			\centering
			\includegraphics[scale=0.4]{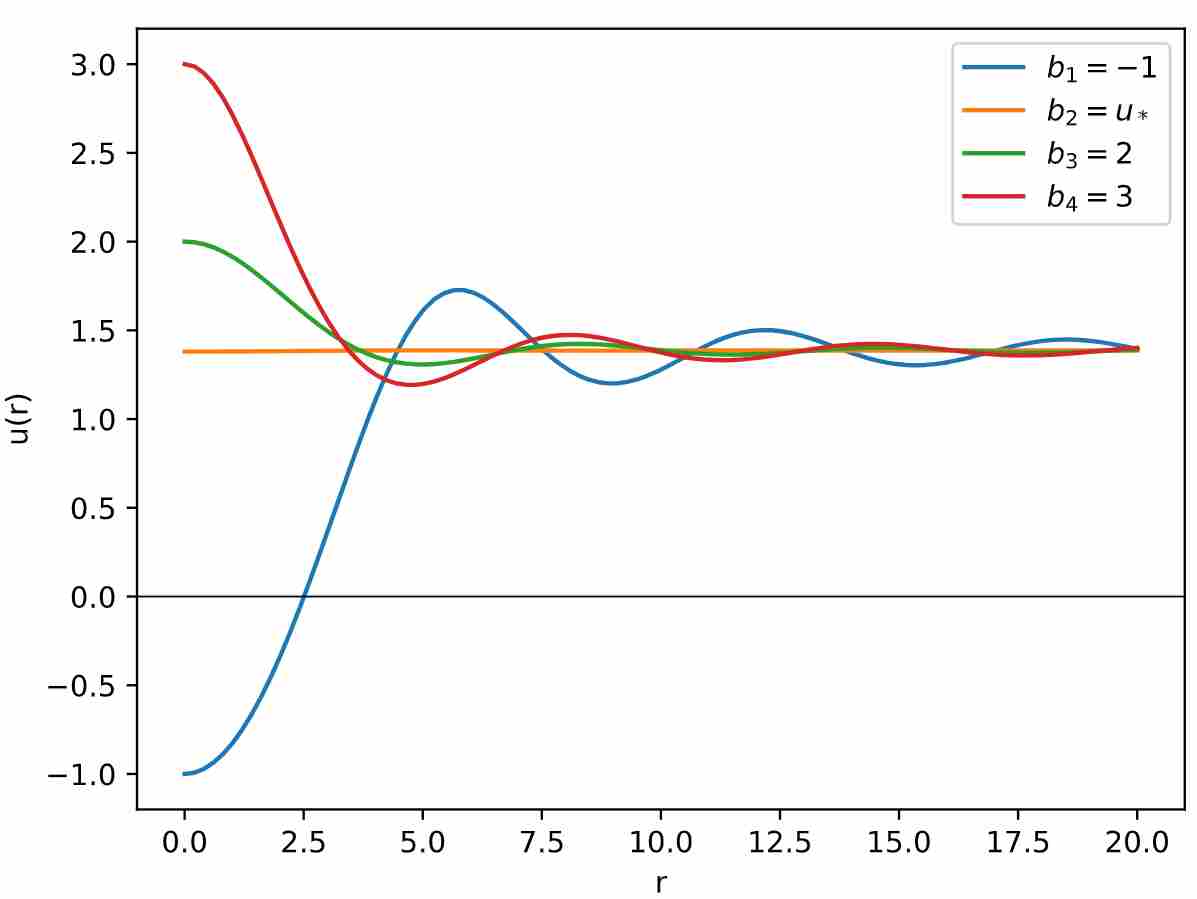}
			\caption{ $\chi=1$, $N = 4$, $m = \frac{1}{2}$, $p=2$, $u_*= 2\ln 2$. Various oscillating behaviours of  $u( \cdot , r)$.}
			\label{fig3:img}
		\end{figure}
    \begin{proof}
    First, we show $R_{max}(b)=\infty$. From Lemma \ref{le3.1}, it is necessary to  obtain the boundedness of $u(r)$. We consider the energy functional
    	\begin{align}
    		\label{3-32}
    		&E\left( r, b\right): =\frac{1}{2}|u^{\prime}\left( r,b \right) |^2+\frac{\chi}{m}e^{mu\left( r,b \right)}-\frac{1}{m}u\left( r,b \right).
    	\end{align}
    	A direct calculation by \eqref{3-6} and \eqref{3-32} yields
    	$$
    	\frac{dE}{dr}(r, b) =-\frac{\left( N-1 \right)}{r}|u^{\prime}\left( r, b \right)| ^2\leq 0.
    	$$
    Clearly,
    $$ E\left( r, b \right) \geq \frac{\chi}{m}e^{mu\left( r,b \right)}-\frac{1}{m}u\left( r,b \right) .$$
    Define
     $$f(u): =\frac{\chi}{m}e^{mu}-\frac{1}{m}u.$$
     It is easy to see that $f(u)$ attains its minimum at $u_*$, and
     $$f_{min}\left( u \right) =f\left( u_* \right) =\frac{\chi}{m}e^{mu_*}-\frac{1}{m}u_*,$$
    where $u_*= \frac{1}{m}\ln \frac{1}{\chi m}$ is the equilibrium point of the equation \eqref{3-6}.
    Then we have
    	$$E\left( r , b\right) \in \left[ \frac{\chi}{m}e^{mu_*}-\frac{1}{m}u_*, E\left( 0, b\right) \right],$$
    which implies $u\left( r \right)$ is bounded. Therefore, $R_{max}(b)=\infty.$
    	
     In what follows, we demonstrate that $u(r)$ oscillates around $u_*$ and $\underset{r\rightarrow \infty}{\lim}\,\,u \left( r \right) =u_*$.

    Without loss of generality, we assume  $b > u_*$  (the case  $b < u_*$  can be handled similarly). By L'Hospital principle and \eqref{3-6}, we see that
    	$$
    	u^{''}\left( 0 \right) =\frac{1}{N}\left( \frac{1}{m}-\chi e^{mb} \right) <0,
    	$$
    	hence $u\left( r \right)$ get local maximum value at $r=0$. It implies that $u\left( r,b \right)$ decreases monotonically for small $r$.

    (i)We show that $u(r, b)$ cannot remain monotonic for all time. Assume for contradiction, that $u(r, b)$ decreases monotonically for all time. Due to the boundedness of $u(r)$, we derive
    	\begin{align}\label{3-33}
    		&\lim_{r\rightarrow \infty} \chi e^{mu}-\frac{1}{m} =0.
    	\end{align}
    	That is,
    	\begin{align}\label{3-34}
    		&\lim_{r\rightarrow \infty} u\left( r \right) =u_*.
    	\end{align}
    	Noticing that
        \begin{align}\label{3-35}
		u^{\prime}(r)=\frac{1}{r^{N-1}}\int_0^{r}{(\frac{1}{m}-\chi e^{mu}) s^{N-1}}ds<0,
	\end{align}
        therefore,
    	$
    	u\left( r \right) \in \left( u_*, b \right]$ , for $\ r\in \left[ 0,\infty \right).
    	$
    	Denote
    	\begin{align}\label{3-36}
    		&v =u-u_*>0.
    	\end{align}
    	Substituting \eqref{3-36} into \eqref{3-6}, we have
    	\begin{align}\label{3-37}
    		&v^{''}+\frac{\left( N-1 \right)}{r}v^{\prime}+\frac{\chi e^{mu}-\frac{1}{m}}{\left( u-u_* \right)}v=0.
    	\end{align}
    	Since
    	$$
    	\frac{\chi e^{mu}-\frac{1}{m}}{\left( u-u_* \right)}=\frac{\chi e^{mu}-\chi e^{mu_*}}{\left( u-u_* \right)} \geq \chi m e^{mu_*}
    	$$
    	and it is well known that all nontrivial solutions of the initial problem
    	\begin{align*}
    		\left\{
    		\begin{aligned}
    			&v^{''}+\frac{\left( N-1 \right)}{r}v^{\prime}+\chi m e^{mu_*}v=0
    			,\\
    			&v\left( 0 \right) =b-u_*, \ v^{\prime}\left( 0 \right) =0, \\
    		\end{aligned}\right.
    	\end{align*}
    	oscillates around the $v=0$, then according to the Strum's comparison theorem in \cite{19}, it is to prove that the nontrivial solutions of \eqref{3-37} also are oscillatory, which is contradict with $v>0$. Thus, $u(r)$ is non-monotone.
    	\par
    	(ii) We will derive that $u(r)$ oscillates around the $u_*,$ where $u_*= \frac{1}{m}\ln \frac{1}{\chi m}.$
    Let $r_1$ is the first minimum point, then from \eqref{3-35}, we see that  $u\left( r_1 \right) < u_*$ (otherwise, $u'\left( r_1 \right) < 0$).
    Similar to the proof in (i), $u\left( r \right)$ cannot remain increasing in $\left( r_1,\infty \right)$. Keep repeating the above process, then we attain two sequences of extreme points,
    	\begin{align*}
    		 \text{minimum points} &\left\{ 0=r_0<r_2<\cdots \right\},
    	\end{align*}
    	and
    	\begin{align*}
    		\text{maximum points} &\left\{ r_1<r_3<\cdots \right\},
    	\end{align*}
    	with $r_i<r_{i+1}$.
    	Noticing the monotonicity of $E\left( r, b\right)$, it is easy to see that
     when $N=1$, the amplitude is constant, that is
    	\begin{align*}
    		&u\left( r_0 \right) = u\left( r_2 \right) = \cdots >u_*
    	\end{align*}
    	and
    	\begin{align*}
    		&u\left( r_1 \right) = u\left( r_3 \right) = \cdots <u_*.
    	\end{align*}
   When $N\geq 2$,
        \begin{align*}
    		&u\left( r_0 \right) > u\left( r_2 \right) > \cdots >u_*
    	\end{align*}
    	and
    	\begin{align*}
    		&u\left( r_1 \right) < u\left( r_3 \right) < \cdots <u_*.
    	\end{align*}
    \end{proof}

  Proposition \ref{pro-2} is a direct consequence of Lemma \ref{le3.9} and Lemma \ref{le3.10}.

  \medskip

	\textbf{ \textit{Proof of Theorem \ref{th-3}.}}  Let $(\rho, c)$ be a backward self -similar solution of \eqref{1-5} with
\begin{align}\label{3-38}
    		&\rho \left( x,t \right) =\left( T-t \right) ^{-\frac{1}{m}}\phi( (T-t)^{-\frac{1}{Nm}}\vert x\vert)  , \quad c\left( x,t \right) =\left( T-t \right) ^{\frac{2-mN}{mN}}\psi((T-t)^{-\frac{1}{Nm}}\vert x\vert ).
    	\end{align}	
Recalling Proposition \ref{pro-1}, \eqref{3-1},  and Remark \ref{re-1}, letting $A_c=a_c^{\frac{p-1}{p-2}}$, $A_i=a_i^{\frac{p-1}{p-2}}$ ( $i=1, 2$),
then $A_1\le A_c\le A_2$. When $q<\frac{Np}{(N-p)_+}-1$ with $p>2$, $A\in \{A_c\}\cup (A_2, +\infty)$, $\phi$ is radially symmetric, decreasing, and compactly supported solution with $\phi(R(A))=\phi'(R(A))=0$, $\phi(r)>0$ for $r<R(A)$.
Consequently, the support of $\rho$ satisfies
\begin{align}\label{3-39}
    		&\vert x\vert <(T-t)^{\frac{1}{Nm}} R(A),
    	\end{align}
implying that  $\rho$  is radially symmetric and compactly supported, and its support shrinks to the origin at the rate $(T-t)^{\frac{1}{Nm}}$, as $\ t \rightarrow T^-.$  From \eqref{3-38} and \eqref{3-39}, we conclude that $\rho$  undergoes finite-time blow-up at $T$, with the origin being the sole blow-up point.
	
    Next, we show that the mass is finite. Define the mass
	$$\mathcal{M} \left( A \right) :=|\partial B_1|\int_0^{R\left( A \right)} \phi (r)r^{N-1}dr.$$
	From Proposition \ref{pro-1}, we  infer that
\begin{align*}
		&\mathcal{M} \left( A\right) = |\partial B_1|\int_0^{R\left( A\right)}{\phi(r)r^{N-1}dr} \leq A|\partial B_1|\left( R\left( A\right) \right) ^N < \infty.
	\end{align*}

	In what follows, we show that for any $f(x) \in C_0(R^N)$ 	,
	$$\int_{R^{N}}{\rho (x,t)f(x)dx} \rightarrow \mathcal{M} (A)f(0), \quad as \quad t \rightarrow T^-.$$
	For any $f(x) \in C_0(R^N)$, we have
	$$
	\int_{R^N}{\rho(x,t)f(x)}dx=\int_{R^N}{(T-t)^{-\frac{1}{Nm}}\phi \left( (T-t)^{-\frac{1}{Nm}}\left| x \right| \right)f(x)}dx=\int_{R^{N}}{\phi(\vert y\vert)f((T-t)^{\frac{1}{Nm}}y)dy}
	$$
    and
    $$
    \int_{R^N}{\rho(x,t)f(0)}dx=\int_{R^{N}}{\phi(\vert y\vert)f(0)dy}=\int_{\vert y \vert \leqslant R(A)}{\phi(\vert y\vert)f(0)dy}=\mathcal{M} (A)f(0).
    $$
	Since $f(x) \in C_0(R^N)$, for any given $\varepsilon$ and a sufficiently large constant $R_0>R(A)$, there exists $T_0<T$ such that when $t \in (T_0,T)$,
	\begin{align*}
		&\sup_{\vert y \vert \leqslant R_0}\vert f((T-t)^\frac{1}{mN}y) -f(0)\vert <\frac{\varepsilon}{\mathcal{M} (A)}.
	\end{align*}
	Hence
	\begin{align*}
		&\Big| \int_{R^{N}}{\rho (x,t)f(x)dx} -\mathcal{M} (A)f(0) \Big|=\Big| \int_{R^{N}}{\phi(\vert y\vert)(f((T-t)^{\frac{1}{Nm}}y)-f(0))dy} \Big|
		\\
		&\leq \Big| \int_{\vert y \vert \leqslant R(A)}{\phi(\vert y\vert)(f((T-t)^{\frac{1}{Nm}}y)-f(0))dy} \Big|
		\\
		& \leq \mathcal{M} (A)\sup_{\vert y \vert \leqslant R(A^*)}\Big| f((T-t)^{\frac{1}{Nm}}y)-f(0) \Big|
		\\
		&<\varepsilon.
	\end{align*}
	It implies that $$\rho \left( x,t \right) \rightarrow \mathcal{M} \left( A\right) \delta \left( x \right) , \, as \, \ t \rightarrow T^-.$$

\section{Forward self-similar singular solution}
In this section, we  study the existence of forward self-similar solutions of \eqref{1-5} with singular initial value
for the ciritial case 	
$$
m=\frac{( p-2 ) N+p}{N}>0\quad \left(\text{i.e.,} \ p=\frac{(m+2)N}{N+1}\right).
$$
{\bf Noticing that $m>0$ implies that $p>\frac{2N}{N+1}$.} Let
\begin{align*}
\rho \left( x,t \right) =t^{-\alpha}\phi (r) ,\quad c\left( x,t \right) =t^{\gamma}\psi(r),
\end{align*}
where $r =t^{-\beta}|x|$, $\alpha$, $\gamma$, $\beta$ are same as those defined in \eqref{2-1}.
Similarly, \eqref{1-5} is transformed into the following initial value problem
\begin{align}
	\label{2-11}\left\{
	\begin{aligned}
		&\left( \frac{\left| \phi ^{\prime}\left( r \right) \right|^{p-2}\phi ^{\prime}\left( r \right)}{\phi \left( r \right)} \right) ^{\prime}+\frac{N-1}{r}\frac{\left| \phi ^{\prime}\left( r \right) \right|^{p-2}\phi ^{\prime}\left( r \right)}{\phi \left( r \right)}+\chi \phi ^m\left( r \right) +\frac{1}{m}
		=0,\\
		&	\phi\left( 0 \right) =A>0, \quad \phi^{\prime}\left( 0 \right) =0,
	\end{aligned}\right.
\end{align}
with
$$
\psi(\xi)=(K*\phi^m)(\xi).
$$
For the fast diffusion case, we establish the existence of ground state solutions to the problem \eqref{2-11} and provide explicit decay rate estimates. In the slow diffusion case, we prove the existence of compactly supported solutions that satisfy the homogeneous Dirichlet-Neumann free boundary condition. The precise statements are given in the following two theorems.

\begin{theorem}\label{th-4}
Assume that $N\geq1$.

$\bullet$ When $\frac{2N}{N+1}<p \leq 2$, for any $A>0$, the problem \eqref{2-11} admits a unique global positive solution $\phi(r) \in C^1[0,\infty)$ and $\phi(r)$ is decreasing such that
$$\underset{r\rightarrow \infty}{\lim}\,\,\phi \left( r \right) =0.$$

$\bullet$ When $p>2$,  for any $A>0$, the problem \eqref{2-11} admits a unique decreasing solution with compact support, which goes to 0 at a finite point $R_A$ with $\phi'(R_A)=0$.
\end{theorem}

\begin{theorem}\label{th-5}
Assume that $N\geq1$. Let $\phi(r)\in C^1[0,\infty)$ be the global solution of \eqref{2-11}. Then when $\frac{2N}{N+1}<p<2$,
		 $$\lim_{r\to\infty}\phi(r)r^{\frac{p}{2-p}}=K^{\frac{p-1}{p-2}},$$
where $K=({\frac{1}{BNm}})^{\frac{1}{p-1}}\frac{p-1}{p}$.
When $p= 2$,
$$
\lim_{r\to\infty}\frac{\ln\phi(r)}{r^2}=-\frac14.
$$
\end{theorem}
Based on Theorem \ref{th-4} and Theorem \ref{th-5}, we can derive the following theorem.

\begin{theorem}[{\bf Forward self-similar solution with expanding compact support and $\delta$-initial data}]
\label{th-6}
Let $\left( \rho ,c \right)$ be the forward self-similar solution of \eqref{1-5} obtained in the above two theorems, i.e.,
\[
\rho \left( x,t \right) =t^{-\frac{1}{m}}\phi \bigl( t ^{-\frac{1}{mN}}\left| x \right| \bigr),\qquad
c\left( x,t \right) = t^{\frac{2-mN}{mN}}\psi \bigl( t ^{-\frac{1}{mN}}\left| x \right| \bigr),
\]
where $m=\frac{(p-2)N+p}{N}$.
Then for any $p>\frac{2N}{N+1}$, the self-similar solution possesses the following properties:

$\bullet$ {\bf Finite mass}. The bacterial density $\rho$ has finite total mass, defined by
\[
M :=\int_{\mathbb{R}^N} \rho(x,t)\,dx = |\partial B_1| \int_{0}^{\infty} \phi(r)\, r^{N-1}\,dr < +\infty .
\]

$\bullet$ {\bf Dirac-$\delta$ initial singularity}. As $t\to 0^{+}$, the solution $\rho$ concentrates at the origin in the sense of distributions, namely
\[
\rho(x,t) \longrightarrow M\,\delta(x), \qquad t\to 0^{+},
\]
which coincides with the initial behavior of the fundamental solution. The finiteness of $M$ is verified as follows:
\begin{itemize}
	\item When $p=2$, $\ln \phi(r)\sim -\dfrac{r^{2}}{4}$, which guarantees $M<+\infty$.
	\item When $\dfrac{2N}{N+1}<p<2$, the asymptotic $\phi(r)\sim K^{\frac{p-1}{p-2}} r^{\frac{p}{p-2}}$ together with the condition $m>0$ implies $\frac{p}{p-2}+N<0$, hence $M<+\infty$.
	\item When $p>2$, the profile $\phi$ has compact support, and
	\[
	M = |\partial B_1| \int_{0}^{R_0} \phi(r)\, r^{N-1}\,dr .
	\]
\end{itemize}

$\bullet$ {\bf Expansion of the support (slow diffusion case $p>2$)}. For $p>2$, the initial datum $\rho$ is a Dirac measure concentrated at the origin, and the support of $\rho$ expands radially as time increases. Specifically, the support is given by
\[
|x| < t^{\frac{1}{mN}} R_0,
\]
so that its radius grows at the rate $t^{\frac{1}{mN}}$.
\end{theorem}
\begin{remark} \label{re4-1}
Noticing that
$$
\psi(\xi)=(K*\phi^m)(\xi),
$$
recalling Theorem \ref{th-5}, one can verify that
$\psi$ is well-defined by \eqref{radial-1} when $p>2\sqrt{\frac{N}{N+1}}$. However, if $\frac{2N}{N+1}<p\le 2\sqrt{\frac{N}{N+1}}$,
$\psi$ becomes
unbounded everywhere, although $\phi'(r)$ remains well-defined, this can be seen from its derivative
$$
\psi'(r)=-r^{1-N}\int_0^rs^{N-1}\phi^m(s)ds.
$$
\end{remark}
To investigate the existence of nonegative solutions to the initial value problem \eqref{2-11}. We begin by performing the following transformation.
Let
\begin{align}\label{4-1}
	&u\left( r \right) =\ln \phi,
\end{align}
for $p= 2$. Then \eqref{2-11} is transformed into
\begin{align}
	\label{4-2}\left\{
	\begin{aligned}
		&u''\left( r \right) +\frac{N-1}{r}u^{\prime}\left( r \right) +\chi e^{um}+\frac{1}{m}=0
		,\\
		&	u\left( 0 \right) =b, \, u^{\prime}\left( 0 \right) =0,
		\\
	\end{aligned}\right.
\end{align}
where $m=\frac 2N$, $b=\ln A$. Let
\begin{align}\label{4-3}
    	&u\left( r \right) =\phi ^{\frac{p-2}{p-1}}\left( r \right),
    \end{align}
and
\begin{align}\label{4-4}
    	&q=\frac{(p-1)m}{p-2},
    \end{align}
for $p\ne 2$. Then when $1<p<2$, \eqref{2-11} is transformed into
    \begin{align}
    	\label{4-5}\left\{
    	\begin{aligned}
    		&\left( B\left| u^{\prime}\left( r \right) \right|^{p-2}u^{\prime}\left( r \right) \right) ^{\prime}+\frac{B\left( N-1 \right)}{r}\left| u^{\prime}\left( r \right) \right|^{p-2}u^{\prime}\left( r \right) -\chi \left| u \right|^{q-1}u-\frac{1}{m}=0
    		,\\
    		&	u\left( 0 \right) =a, u^{\prime}\left( 0 \right) =0.\\
    	\end{aligned}\right.
    \end{align}
    When $p>2$, \eqref{2-11} is transformed into
    \begin{align}
    	\label{4-6}\left\{
    	\begin{aligned}
    		&\left( B\left| u^{\prime}\left( r \right) \right|^{p-2}u^{\prime}\left( r \right) \right) ^{\prime}+\frac{B\left( N-1 \right)}{r}\left| u^{\prime}\left( r \right) \right|^{p-2}u^{\prime}\left( r \right) +\chi \left| u \right|^{q-1}u+\frac{1}{m}=0
    		,\\
    		&	u\left( 0 \right) =a, u^{\prime}\left( 0 \right) =0.\\
    	\end{aligned}\right.
    \end{align}
     where $B=| \frac{p-1}{p-2} | ^{p-1}$ and $a=A^{\frac{p-2}{p-1}}.$
     For later convenience in the proof, we allow the solution to take negative values, and thus replace $u^{q}$ with $ |u|^{q-1}u$.
Following a similar argument as in Lemma  \ref{le3.1}, he local existence of solutions to the problem \eqref{4-2} (or \eqref{4-6}) can be established directly.

We first analyze the case $p=2$  and investigate the existence of solutions to the problem \eqref{4-2}. Based on \eqref{4-1}, we expect to find a global solution of the problem \eqref{4-2} such that $\underset{r\rightarrow \infty}{\lim}\,\,u \left( r \right) =-\infty$,
 which ensures that $\phi(r)$ remains a globally positive solution with $\underset{r\rightarrow \infty}{\lim}\,\,\phi \left( r \right) =0$.
\begin{lemma}\label{le4.1}
Let $u\left( r \right)$ be a classical solution of \eqref{4-2}, with $\left( 0, R_{max}(b) \right)$ denoting its maximal existence interval.
For any $b\in\mathbb R$, we have $R_{max}(b) = \infty$ and $u\left( r \right)$ is decreasing and satisfying
	 $$\underset{r\rightarrow \infty}{\lim}\,\,u \left( r \right) =-\infty.$$
	 \end{lemma}
	 \begin{proof}
	 First, we show that $R_{max}=\infty$ and $u\left( r \right)$ is decreasing on $(0,\infty)$.
    From \eqref{4-2}, we see that
    $$(r^{N-1}u')^{'} =-(\chi e^{um}+\frac{1}{m})r^{N-1}.$$
    Integrating this equality from 0 to $r$ yields
    \begin{align}\label{4-7}
    	\begin{aligned}
    		&r^{N-1}u'=-\int_0^{ r }{(\chi e^{um}+\frac{1}{m})s^{N-1}ds},
    	\end{aligned}
    \end{align}
    then $u'(r)\leq-\frac{1}{mN}r<0$ on $(0, R_{max})$, which implies that $u(r)$ is monotonically decreasing on $(0, R_{max})$.

    We claim that $R_{max}=\infty.$ Suppose the contrary, $R_{max}<\infty$, and $\underset{r\rightarrow R_{max}^{-}}{\lim}u \left( r \right) = -\infty$ since $u(r)$ is monotonically decreasing on $(0, R_{max})$.
    By \eqref{4-7}, we see that
    \begin{align*}
    	\begin{aligned}
    		&r^{N-1}u'(r) \geq -\int_0^{ r }{(\chi e^{bm}+\frac{1}{m})s^{N-1}ds}=-\frac{1}{N}(\chi e^{bm}+\frac{1}{m})r^{N}.
    	\end{aligned}
    \end{align*}
    That is,
   \begin{align*}
    	\begin{aligned}
    		&u'(r) \geq -\frac{1}{N}(\chi e^{bm}+\frac{1}{m})r.
    	\end{aligned}
    \end{align*}
    Integrating this inequality from 0 to $r$ yields
    \begin{align*}
    	\begin{aligned}
    		&u(r) \geq b -\frac{1}{2N}(\chi e^{bm}+\frac{1}{m})r^2,
    	\end{aligned}
    \end{align*}
    which implies that $\underset{r \rightarrow R_{max}^{-}}{\lim}\,\,u \left( r \right) > -\infty.$ It is a contradiction. Therefore, $R_{max}=\infty.$
	
 Next, we demonstrate that $\underset{r\rightarrow \infty}{\lim}\,\,u \left( r \right) =-\infty$. Suppose the contrary,  then there exist a constant $C_1$ such that $\underset{r \rightarrow \infty}{\lim}\,\,u \left( r \right) = C_1$. Using \eqref{4-7}, we see that
    \begin{align*}
    	\begin{aligned}
    		&r^{N-1}u'(r) \leq -\int_0^{ r }{(\chi e^{mC_1}+\frac{1}{m})s^{N-1}ds}=-\frac{1}{N}(\chi e^{mC_1}+\frac{1}{m})r^{N}.
    	\end{aligned}
    \end{align*}
    That is,
 \begin{align*}
    	\begin{aligned}
    		&u'(r) \leq-\frac{1}{N}(\chi e^{mC_1}+\frac{1}{m})r.
    	\end{aligned}
    \end{align*}
    Integrating this inequality from 0 to $r$ yields
    \begin{align*}
    	\begin{aligned}
    		&u(r) \leq b -\frac{1}{2N}(\chi e^{mC_1}+\frac{1}{m})r^2,
    	\end{aligned}
    \end{align*}
    which implies that $\underset{r \rightarrow \infty}{\lim}\,\,u \left( r \right) \leq -\infty.$ It is a contradiction. Therefore, $\underset{r\rightarrow \infty}{\lim}\,\,u \left( r \right) =-\infty$.
\end{proof}

Next, we consider the case $1<p<2$. By \eqref{4-3} and $\frac{p-2}{p-2}<0$, we expect to find a global positive solution of the problem \eqref{4-5} such that $\underset{r\rightarrow \infty}{\lim}\,\,u \left( r \right) =+\infty$ in order to obtain that $\phi(r)$ is a global positive solution and $\underset{r\rightarrow \infty}{\lim}\,\,\phi \left( r \right) =0.$

\begin{lemma}\label{le4.2}
	 Assume that $1<p<2$, $q<0$. Let $u\left( r \right)$ be a classical solution of \eqref{4-5}, and $\left( 0, R_{max}(a)\right)$ be the maximal existence interval of the solution. Then $R_{max}(a)=\infty$ for any $a>0$, and $u\left( r \right)$ is increasing, satisfying
$$\underset{r\rightarrow \infty}{\lim}\,\,u \left( r \right) =+\infty,$$
and
$$
\lim_{r\to\infty} \frac{u(r)}{r^{\frac{p}{p-1}}}=\frac{p-1}{p}(\frac{1}{mBN})^\frac{1}{p-1}.
$$
	 \end{lemma}
\begin{proof}

 First, we show that $R_{max}=\infty$ and $u\left( r \right)$ is increasing on $(0,\infty)$.
 From \eqref{4-5}, it is easy to see that
     \begin{align*}
     	& \left( Br^{N-1}\left| u^{\prime}\left( r \right) \right|^{p-2}u^{\prime}\left( r \right) \right) ^{\prime}=\left( \frac{1}{m}+\chi \left| u \right|^{q-1}u \right) r^{N-1}.
     \end{align*}
     Integrating this equality from 0 to $r$ yields
     \begin{align}\label{4-8}
    	& Br^{N-1}\left| u^{\prime}\left( r \right) \right|^{p-2}u^{\prime}\left( r \right) =\int_0^r{\left( \frac{1}{m}+\chi \left| u \right|^{q-1}u \right) s^{N-1}ds}.
    \end{align}
     Since $u(0)=a>0$, then $u(r)$ is monotonically increasing in a small right neighborhood of 0. We claim that $u(r)$ is increasing on $(0,R_{max})$.
     Suppose the contrary, then there exists $r^*$ (the first maximum piont), such that $u(r)$ is increasing on $[0,r^*]$, and
     $$ u'(r^*)=0, \quad \left(\left| u^{\prime}\left( r^* \right) \right|^{p-2}u^{\prime}\left( r^* \right) \right) ^{\prime} \leq 0.$$
     But by \eqref{4-5}, we have
     $$ 0 \geq  B\left(\left| u^{\prime}\left( r^* \right) \right|^{p-2}u^{\prime}\left( r^* \right) \right) ^{\prime} = \frac{1}{m}+\chi \left| u \right|^{q-1}u > \frac{1}{m} >0.$$
     It is a contradiction. Therefore, $u(r)$ is increasing on $(0,R_{max})$.

 We assert that $R_{max}=\infty$. Suppose the contrary, $R_{max}<\infty$. Since $u(r)\geq a>0$ for $r \in (0, R_{max})$, then Lemma \ref{le3.1} also holds for $1 < p< 2$. By Lemma \ref{le3.1}, we attain that $\underset{r\rightarrow R_{max}^{-}}{\lim}u \left( r \right) = +\infty$
 since $u(r)$ is increasing on $(0,R_{max})$.
    Due to $q <0$ and \eqref{4-8}, we see that
    \begin{align*}
    	& Br^{N-1}\left| u^{\prime}\left( r \right) \right|^{p-2}u^{\prime}\left( r \right) \leq \int_0^r{( \frac{1}{m}+\chi a^{q}) s^{N-1}ds}\leq \frac{1}{N}( \frac{1}{m}+\chi a^{q}) r^{N}.
    \end{align*}
     That is,
     \begin{align*}
    	&  u^{\prime}(r)\leq (\frac{1}{mBN} + \frac{\chi}{BN} a^{q})^\frac{1}{p-1} r^{\frac{1}{p-1}}.
    \end{align*}
    Integrating this inequality from 0 to $r$ yields
    \begin{align*}
    	&  u(r)\leq a + \frac{p-1}{p}(\frac{1}{mBN} + \frac{\chi}{BN} a^{q})^\frac{1}{p-1} r^{\frac{p}{p-1}},
    \end{align*}
    which implies that $\underset{r \rightarrow R_{max}^{-}}{\lim}\,\,u \left( r \right) < +\infty.$ It is a contradiction. Therefore, $R_{max}=\infty.$

    Next, we will show that $\underset{r\rightarrow \infty}{\lim}\,\,u \left( r \right) =+\infty.$
    Recalling \eqref{4-8}, we see that
\begin{align*}
Br^{N-1}\left| u^{\prime}\left( r \right) \right|^{p-2}u^{\prime}\left( r \right) \geq \int_0^r{( \frac{1}{m}+\chi C_2^{q}) s^{N-1}ds}\geq \frac{1}{mN} r^{N}.
\end{align*}
That is,
\begin{align*}
u^{\prime}(r)\geq (\frac{1}{mBN})^\frac{1}{p-1} r^{\frac{1}{p-1}}.
\end{align*}
Integrating this inequality from 0 to $r$ yields
\begin{align}\label{as-1}
u(r)\geq a + \frac{p-1}{p}(\frac{1}{mBN})^\frac{1}{p-1} r^{\frac{p}{p-1}}\ge  \frac{p-1}{p}(\frac{1}{mBN})^\frac{1}{p-1} r^{\frac{p}{p-1}}
\end{align}
for any $r>0$,  which implies that $\underset{r\rightarrow \infty}{\lim}\,\,u \left( r \right) =+\infty.$
    Noticing that $q<0$, and substituting \eqref{as-1} into \eqref{4-8} yields
      \begin{align*}
    	& Br^{N-1}\left| u^{\prime}\left( r \right) \right|^{p-2}u^{\prime}\left( r \right) =\int_0^1 \left( \frac{1}{m}+\chi u^q \right) s^{N-1}ds
    +\int_1^r \left( \frac{1}{m}+\chi u^q \right) s^{N-1}ds
    \\
    \le &\left( \frac{1}{m}+\chi a^q \right)+ \int_1^r \left( \frac{1}{m} s^{N-1}+\chi \frac{(p-1)^q}{p^q}(\frac{1}{mBN})^\frac{q}{p-1} s^{\frac{pq}{p-1}+N-1}  \right)ds
    \\
    \le &\left( \frac{1}{m}+\chi a^q \right)+\frac{1}{mN} (r^{N}-1)+ \chi \frac{(p-1)^q}{p^q}(\frac{1}{mBN})^\frac{q}{p-1} \frac{1}{\frac{pq}{p-1}+N} \left(r^{\frac{pq}{p-1}+N}-1\right)
    \\
    \le & \left( \frac{1}{m}+\chi a^q \right)+\frac{1}{mN} r^{N}+ \chi \frac{(p-1)^q}{p^q}(\frac{1}{mBN})^\frac{q}{p-1} \frac{1}{\frac{pq}{p-1}+N} r^{\frac{pq}{p-1}+N}
    \end{align*}
    for any $r\ge 1$. It implies that
   \begin{align*}
  u^{\prime}
    \le & \left(\left( \frac{1}{mB}+\frac{\chi}{B} a^q \right)r^{1-N}+\frac{1}{mNB} r+ \chi \frac{(p-1)^q}{Bp^q}(\frac{1}{mBN})^\frac{q}{p-1} \frac{1}{\frac{pq}{p-1}+N} r^{\frac{pq}{p-1}+1}\right)^{\frac1{p-1}}.
    \end{align*}
    Integrating the above inequality from $1$ to $r$ gives
     \begin{align}\label{as-2}
  u(r)\le u(1)+\int_1^r \left(\left( \frac{1}{mB}+\frac{\chi}{B} a^q \right)s^{1-N}+\frac{1}{mNB} s+ \chi \frac{(p-1)^q}{Bp^q}(\frac{1}{mBN})^\frac{q}{p-1} \frac{1}{\frac{pq}{p-1}+N} s^{\frac{pq}{p-1}+1}\right)^{\frac1{p-1}}ds.
    \end{align}
    Combining \eqref{as-1} and \eqref{as-2}, applying ${\rm L^{\prime}H\hat opital^{\prime}s}$ rule yields
    $$
    \lim_{r\to\infty} \frac{u(r)}{r^{\frac{p}{p-1}}}=\frac{p-1}{p}(\frac{1}{mBN})^\frac{1}{p-1}.
    $$
    We complete the proof of this lemma.
\end{proof}

In what follows, we consider the case $p>2$. By \eqref{4-3} and $\frac{p-2}{p-1}>0$, we expect to find a global solution of the problem \eqref{4-6} such that $\underset{r\rightarrow \infty}{\lim}\,\,u \left( r \right) =0$ for which $\phi(r)$ is a global solution and $\underset{r\rightarrow \infty}{\lim}\,\,\phi \left( r \right) =0.$ Then there is no such solution. That is,
   \begin{lemma}\label{le4.3}
	 Assume that $p>2, q>p-1$. Let $u\left( r \right)$ be a classical solution of \eqref{4-6}, and $\left( 0, R_{max} \right)$ be the maximal existence interval of the solution. Then for any $a>0$, $R_{max}(a)=\infty$ and there exists a point $R_0(a)$ such that $u(r)$ is decreasing in $(0, R_0(a))$, and
	 $$\underset{r\rightarrow R_0^{-}}{\lim}\,\,u \left( r \right) =0, \quad \text{and} \quad \underset{r\rightarrow R_0^-}{\lim}\,\,u' \left( r \right) <0.$$
	 \end{lemma}
\begin{proof}

First, we will show that $R_{max}=\infty$.  Similar to Lemma \ref{le3.1}, we only need to establish the boundedness of $u(r)$. Denote the energy functional
\begin{align}\label{4-9}
		E\left( r, a \right): =\frac{B\left( p-1 \right)}{p}\left| u^{\prime}\left( r,a \right) \right|^p+ \frac{\chi}{q+1}|u(r,a) |^{q+1}+\frac{1}{m}u(r,a),
		\end{align}
with
$$
E(0,a)= \frac{\chi}{q+1}a^{q+1}+\frac{a}{m}.
$$
From \eqref{4-9}, a direct calculation  yields
		$$\frac{dE}{dr}\left( r,a \right) =-\frac{B\left( N-1 \right)}{r}\left| u^{\prime}\left( r,a \right) \right|^p\leq 0,$$
then $E(r,a)$ is decreasing. By \eqref{4-9}, we see that
 $$\frac{\chi}{q+1}|u |^{q+1}+\frac{1}{m}u \leq E(0,a),$$
 which implies that $u(r,a)$ is bounded since $q>p-1$. Then $R_{max}=\infty$.

Next, we show that $u(r)$ reaches 0 at a finite point $R_0$.
From \eqref{4-6}, it is easy to see that
     \begin{align*}
     	& \left( Br^{N-1}\left| u^{\prime}\left( r \right) \right|^{p-2}u^{\prime}\left( r \right) \right) ^{\prime}=-\left( \frac{1}{m}+\chi \left| u \right|^{q-1}u \right) r^{N-1}.
     \end{align*}
     Integrating this equality from 0 to $r$ yields
     $$
 Br^{N-1}\left| u^{\prime}\left( r \right) \right|^{p-2}u^{\prime}\left( r \right) =-\int_0^r{\left( \frac{1}{m}+\chi \left| u \right|^{q-1}u \right) s^{N-1}ds}.
$$
It implies that $u(r)$ is decreasing when $u\ge 0$, hence,
$$
 Br^{N-1}\left| u^{\prime}\left( r \right) \right|^{p-2}u^{\prime}\left( r \right)\le -  \frac{1}{BmN}r^N,
$$
or equivalently,
 \begin{align}\label{4-10}
     	u^{\prime}\le -  (BmN)^{-\frac{1}{p-1}}r^{\frac{1}{p-1}}.
 \end{align}
Consequently,
\begin{align*}
     	u(r)\le a -  (BmN)^{-\frac{1}{p-1}}\int_0^rs^{\frac{1}{p-1}}ds=a -  \frac{p-1}p(BmN)^{-\frac{1}{p-1}}r^{\frac{p}{p-1}},
\end{align*}
which shows there exists a point $R_0$ such that
$u(R_0)=0$. From \eqref{4-10}, it follows immediately that $u'(R_0)<0$. The proof is complete.
\end{proof}

\medskip

{\it\bf Proof of Theorem \ref{th-4}.}  By \eqref{4-1} and Lemma \ref{le4.1}, we see that when $p=2$, $\phi\left( r \right)$ is decreasing and fulfills $\underset{r\rightarrow \infty}{\lim}\,\,\phi\left( r \right) =0.$ From \eqref{4-3} and Lemma \ref{le4.2}, we deduce that when $1<p<2$, $\phi\left( r \right)$ is also decreasing and satisfies $\underset{r\rightarrow \infty}{\lim}\,\,\phi\left( r \right) =0.$
From Lemma \ref{le4.3} and \eqref{4-3}, when $p>2$, there exist a point $R_0$ such that $\underset{r\rightarrow R_0}{\lim}\,\,\phi\left( r \right) =0$ and
$$\lim_{r\to R_0^-}\phi^{'}(r)=\lim_{r\to R_0^-}\frac{p-1}{p-2}u^{\frac{1}{p-2}}u'(r) =0.$$

With these, the proof of Theorem \ref{th-4} is complete. \hfill $\Box$

  \medskip

Next, we will further derive the decay rate at infinity for $1<p\leq 2$.

\begin{lemma}\label{le4.4} Assume that $1<p\le 2$.
Let $\phi\left( r \right)\in C^1[0,+\infty)$ be the global solution of \eqref{2-11}. If $1<p<2$ and $q<0$, then
		\begin{align*}
		\lim_{r\to\infty}\phi(r)r^{\frac{p}{2-p}} = K^{\frac{p-1}{p-2}},
		\end{align*}
		where $K=({\frac{1}{BNm}})^{\frac{1}{p-1}}\frac{p-1}{p}.$
		If $p=2$, then
		\begin{align*}
		\lim_{r\to\infty} \frac{\ln \phi(r)}{r^2} =-\frac14.
		\end{align*}
	\end{lemma}

	\begin{proof}
When $1<p<2$, using Lemma \ref{le4.2}, we see that
	 $$u(r) \sim Kr^{\frac{p}{p-1}} \quad as \quad  r\rightarrow \infty,$$
where $K=({\frac{1}{BNm}})^{\frac{1}{p-1}}\frac{p-1}{p}.$ By \eqref{4-3}, we can see
	$$\phi(r) \sim K^{\frac{p-1}{p-2}}r^{\frac{p}{p-2}} \quad as \quad  r\rightarrow \infty.$$

	When $p=2$, denote
	$$H(r)=\int_0^{ r }{(\chi e^{um}+\frac{1}{m})s^{N-1}ds}.$$
	Noticing that $u(r)\to-\infty$ as $r\to\infty$, therefore
	$H(r) \sim \frac{1}{mN}r^{N}$  as $ r\rightarrow \infty$.
	By \eqref{4-7}, it is easy to see that $u'(r)=-\frac{H(r)}{r^{N-1}}$, then $u'(r) \sim -{\frac{r}{Nm}}$ as $ r\rightarrow \infty$. By a direct calculation, it is not difficult to derive that
	$$u(r) \sim -\frac{r^{2}}{2Nm} \quad as \quad  r\rightarrow \infty.$$
	Noticing that $mN=2$ for $p=2$, we complete the proof.
	\end{proof}

\medskip
According to the result of Lemma \ref{le4.4}, Theorem \ref{th-5} is obtained directly.
\medskip

\textbf{ \textit{Proof of Theorem \ref{th-6}.}}
  Recall that the $(\rho, c)$ is the forward self -similar solution of \eqref{1-5}. That is,
  $$
		\rho \left( x,t \right) =t  ^{-\frac{1}{m}}\phi \left( t  ^{-\frac{1}{mN}}\left| x \right| \right) ,\quad c\left( x,t \right) = t  ^{\frac{2-mN}{mN}}\psi \left( t  ^{-\frac{1}{mN}}\left| x \right| \right).
		$$
When $1<p\le 2$, we show that $$ \rho(x,t) \rightarrow M\delta(x), \quad t \rightarrow 0^{+}$$ 		in the sense of distribution.
For any $f(x) \in C_0(R^N)$, we have
			$$
			\int_{R^N}{\rho(x,t)f(x)}dx=\int_{R^N}{t  ^{-\frac{1}{m}}\phi \left( t  ^{-\frac{1}{mN}}\left| x \right| \right)f(x)}dx=\int_{R^{N}}{\phi(\vert y\vert)f(t^{\frac{1}{Nm}}y)dy} .
			$$
			Then
			\begin{align}\label{4-12}
			&\Big| \int_{R^{N}}{\rho (x,t)f(x)dx} -Mf(0) \Big|=\Big| \int_{R^{N}}{\phi(\vert y\vert)(f(t^{\frac{1}{Nm}}y)-f(0))dy} \Big| \notag
			\\
			&\leq \Big| \int_{\vert y \vert \leqslant R}{\phi(\vert y\vert)(f(t^{\frac{1}{Nm}}y)-f(0))dy} \Big| + \Big| \int_{\vert y \vert > R}{\phi(\vert y\vert)(f(t^{\frac{1}{Nm}}y)-f(0))dy} \Big| \notag
			\\
			&\leq M\sup_{\vert y \vert \leqslant R}\Big| f(t^{\frac{1}{Nm}}y)-f(0) \Big|+ 2\vert\vert f \vert\vert_{L^\infty}  \Big| \int_{\vert y \vert > R}{\phi(\vert y\vert)dy} \Big|.
		\end{align}
			Since $M<\infty$, for any given $\varepsilon$ and there exists a sufficiently large constant $R_0>0$, such that
			\begin{align}\label{4-13}
				& |\partial B_1|\int_{R_0}^{\infty}{\phi(r)}r^{N-1}dr < \frac{\varepsilon}{4\vert\vert f \vert\vert_{L^\infty}}.
			\end{align}
			For the above established $R_0$, there exists $t_0>0$ such that when $t < t_0$,
		\begin{align}\label{4-14}
			&\sup_{\vert y \vert \leqslant R_0}\vert f(t^{\frac{1}{Nm}}y) -f(0)\vert <\frac{\varepsilon}{2M}.
		\end{align}
	 Combining  \eqref{4-12}, \eqref{4-13} and \eqref{4-14}, we obtain that for any $\varepsilon$, there exists $t_0>0$, such that when $t<t_0$,
	 \begin{align*}
	 	\Big| \int_{R^{N}}{\rho (x,t)f(x)dx} -Mf(0) \Big| < \varepsilon.
	 \end{align*}
		It implies that $$\rho \left( x,t \right) \rightarrow M \delta \left( x \right) , \quad as \quad \ t \rightarrow 0^+.$$
When $p>2$, the proof is similar to above, we omit it.
In particular,  when $p>2$, we observe that the function $\phi(r)$ is compactly supported, with its support given by
$$|x|<t^{\frac1{mN}} R_0.$$
This completes the proof of the theorem. \hfill $\Box$

\section{A brief discussion on the sign structure of the chemical concentration in parabolic-elliptic Keller-Segel model}

The parabolic-elliptic Keller-Segel model describes chemotactic aggregation through the coupled system
\begin{align*}\left\{\begin{aligned}
&\rho_t - \Delta_p \rho = -\chi\nabla \cdot (\rho \nabla c),
\\
& -\Delta c = \rho^m,\end{aligned}\right.
\end{align*}
where the chemical concentration $c$ is determined by the fundamental solution of the Poisson equation. The sign structure of $c$ exhibits a remarkable dimensional dependence that profoundly influences the aggregation dynamics.  We note that in one- and two-dimensional spatial settings, the variable $c$ may take negative values. If $c$ is interpreted as a gravitational potential in  a model describing the gravitational interaction between particles \cite{BH, Wo}, negative values are permissible. However, if $c$ represents the concentration of a chemical substance, a negative concentration appears to lack physical meaning. In this context, on the one hand,  $c$  can be understood as a deviation from a certain reference concentration \cite{13}; on the other hand, we attempt to provide an alternative interpretation: the model may actually describe a chemotactic mechanism that potentially incorporates both attractive and repulsive effects simultaneously, as detailed below.

In dimensions $N \geq 3$, the solution of the Poisson equation is given by
$$
c(x) = \frac{1}{(N-2)\omega_N} \int_{\mathbb{R}^N} \frac{1}{|x-y|^{N-2}} \rho^m(y) dy,
$$
where $\omega_N$ is the surface area of the unit sphere in $\mathbb{R}^N$. Since the kernel is positive, the chemical concentration $c$ is always non-negative, leading to purely attractive interactions that promote cell aggregation. This positive-definite nature of the kernel in higher dimensions is a key factor contributing to the possibility of finite-time blow-up in three-dimensional and higher-dimensional Keller–Segel systems.

In two dimensions, the fundamental solution takes the logarithmic form, then
$$
c(x,t) = -\frac{1}{2\pi} \int_{\mathbb{R}^2} \ln |x-y|\,\rho^m(y,t)\,dy.
$$
The logarithmic kernel introduces a sign-changing property: for $|x-y|<1$, $\ln|x-y| < 0$, making the integrand positive and contributing to a local increase in $c$. This corresponds to the release of chemoattractant signals by cells detecting favorable conditions in their immediate vicinity, promoting aggregation. Conversely, for $|x-y|>1$, the integrand becomes negative, contributing to a decrease in $c$. This can be interpreted as a repulsive signal that discourages cell aggregation in resource-poor regions. The sign structure of $c$ thus provides a spatially modulated communication mechanism: positive signals attract cells toward resource-rich zones, while negative signals suppress migration toward unfavorable areas, enabling self-organized collective behavior.

In one dimension, the chemical concentration is given by
$$
c(x,t) = -\frac{1}{2} \int_{\mathbb{R}} |x-y|\,\rho^m(y,t)\,dy.
$$
Unlike the two-dimensional case, the absolute value function $|x-y|$ is always non-negative, and the presence of the negative sign ensures that $c(x,t) \leq 0$ for all $x$ and $t$. This implies that the chemical signal is always repulsive in nature. One can interpret this by defining $h = -c \geq 0$ as a repulsive sign concentration, where cells tend to move away from regions of high $h$, i.e., from high $h$ regions (far from the center) toward low $h$ regions (the center). For example, under radial symmetry, the Poisson equation reduces to
$$
h_{rr} = \rho^m(r),
$$
which, after integration with the symmetry condition $h_r(0)=0$, yields
$h_r = \int_0^r \rho^m(s)\,ds \geq 0$, so $h$ is lowest at the center and grows outward.

In summary, for one-dimensional, two-dimensional, and higher-dimensional cases, although the sign of $c$ may vary across dimensions, the direction of its gradient $\nabla c$ remains unchanged and
 always points toward the signal source (i.e., in the direction of $y$). Specifically, $\nabla c$ can be uniformly expressed in the following form,
$$
\nabla c(x,t)=-\frac{1}{\omega_N}\int_{\mathbb R^N}\frac{(x-y)}{|x-y|^N}\rho^m(y)dy.
$$
The transition from purely attractive interactions in dimensions $n \geq 3$ to sign-changing behavior in two dimensions and purely repulsive interactions in one dimension  underscores the critical role of spatial dimension and kernel structure in chemotactic systems. The one-dimensional case, in particular, reveals that a signal that is repulsive in the sense of $h = -c$ can still produce center-ward aggregation, with a drift strength that grows with distance from the center.
	
\section*{Declarations}
The authors  declare  that there is no competing interest. Data sharing is not applicable to this article as no new data were created or analyzed in this study.

\section*{Acknowledgments}
  This work is supported by NSFC(12271186).

\end{document}